\documentclass[11pt]{article}

\usepackage[utf8]{inputenc}
\usepackage[T1]{fontenc}
\usepackage{lmodern}
\usepackage{amsmath,amssymb,amsthm,mathtools}
\usepackage{geometry}
\usepackage{enumitem}
\usepackage[hidelinks]{hyperref}
\usepackage{microtype}

\newtheorem{theorem}{Theorem}[section]
\newtheorem{proposition}[theorem]{Proposition}
\newtheorem{lemma}[theorem]{Lemma}
\newtheorem{corollary}[theorem]{Corollary}

\theoremstyle{definition}
\newtheorem{definition}[theorem]{Definition}
\newtheorem{problem}[theorem]{Problem}
\theoremstyle{remark}
\newtheorem{remark}[theorem]{Remark}

\newcommand{\Hh}{\mathbb H}
\newcommand{\Cc}{\mathbb C}
\newcommand{\Rr}{\mathbb R}
\newcommand{\Ss}{\mathbb S}
\newcommand{\Zz}{\mathbb Z}
\newcommand{\Sp}{\operatorname{Sp}}
\newcommand{\conv}{\operatorname{conv}}
\newcommand{\rank}{\operatorname{rank}}
\newcommand{\cS}{\mathcal S}
\newcommand{\cZ}{\mathcal Z}

\title{\textbf{Quaternionic LVM Manifolds}}
\author{Alberto Verjovsky\\[2mm]
\small Instituto de Matem\'aticas, Universidad Nacional Aut\'onoma de M\'exico\\
\small Unidad Cuernavaca, M\'exico\\
\small \texttt{albertoverjovsky@gmail.com}}

\date{}

\begin{document}
\maketitle

\begin{center}
\emph{Dedicated to the memory of Michael ``Misha'' Kapovich\\
(March 13, 1963--June 16, 2026)}
\end{center}
\medskip

\begin{abstract}
An admissible configuration $\Lambda\subset\Hh^m\simeq\Rr^{4m}$
defines a compact intersection of quadrics $Z_{\Hh}(\Lambda)$ and a
free diagonal $\Sp(1)$ quotient $N_{\Hh}(\Lambda)$. We realize this
quotient as the leaf space of a real dilation action on an open subset
of quaternionic projective space. The total space is a quaternionic
moment-angle polyhedral product. Every quotient is $2$-connected; when
there are no indispensable coordinates, it is $3$-connected and its
fourth integral cohomology is generated by the Euler class $e$ of the
associated oriented rank-four bundle. We compute
$p_1(TN_{\Hh})=2(n-2)e$ and prove that no LVMQ manifold admits a
quaternionic K\"ahler metric. An explicit infinite family consists of
products of spheres admitting both homogeneous positive Einstein
metrics and quaternionic toric structures in the sense of
Gentili--Gori--Sarfatti. We also describe the trace foliations of the
real action in the Poincar\'e domain. For the separate rank-one
distribution $X\Hh$, an invertible quaternionic-linear field is
integrable only in the real scalar case. A nonlinear field with
positive real scalar derivative at the origin and involutive
$X\Hh$ has exactly the standard quaternionic Hopf foliation on every
sufficiently small centered sphere.
\end{abstract}



\section{Introduction}\label{sec:intro}

LVM manifolds arise as compact quotients of open subsets of complex
projective space by holomorphic actions \cite{LVM,Meer,Verjovsky}.
Their description by intersections of real quadrics relates their
topology to convex polytopes and moment-angle manifolds \cite{BM,BP}.
These constructions include Hopf and Calabi--Eckmann examples. Under
the rationality condition on the configuration, the canonical foliation
has compact complex-torus leaves and a projective toric orbifold as
leaf space. Conversely, every projective toric variety with at most
quotient singularities occurs as such a leaf space \cite{MV04,MV08}.
Irrational configurations lead to the non-commutative toric geometry
of Katzarkov, Lupercio, Meersseman and Verjovsky \cite{KLMV14,KLMV21}.

For quaternionic coordinates, the norm-minimization argument still
uses real convex geometry. Put $p=4m$, with $m\ge1$, and consider a
configuration
\[
\Lambda=(\lambda_1,\ldots,\lambda_n)\subset\Rr^p,
\qquad n>p.
\]
We call it admissible when the origin belongs to its convex hull but
to the convex hull of no subconfiguration with at most $p$ vectors.
Define
\[
Z_{\Hh}(\Lambda)=
\left\{q\in\Hh^n:
\sum_i\lambda_i|q_i|^2=0,\quad \sum_i|q_i|^2=1\right\}.
\]
The diagonal left action of $\Sp(1)$ is free. Its quotient
\[
N_{\Hh}(\Lambda)=Z_{\Hh}(\Lambda)/\Sp(1)
\]
is a smooth manifold of dimension $4(n-m-1)$, which we call an LVMQ
manifold. This term refers to the quotient construction with
quaternionic coordinates.

\subsection*{The quotient and its topology}

The coordinatewise right $\Sp(1)^n$ action on $Z_{\Hh}(\Lambda)$ has
orbit space
\[
P_\Lambda=\left\{x\in\Rr^n_{\ge0}:
\sum_i\lambda_i x_i=0,\quad\sum_i x_i=1\right\}.
\]
This is a simple polytope of dimension $n-4m-1$. A coordinate is
\emph{indispensable} if it is positive everywhere on $P_\Lambda$.
In the facet complex $K_{P_\Lambda}$ it is a ghost vertex, meaning an
index whose singleton is not a simplex. With these indices retained,
\[
Z_{\Hh}(\Lambda)\cong (D^4,\Ss^3)^{K_{P_\Lambda}}.
\]
The polyhedral product on the right is the union of products having
$D^4$ in the coordinates of a simplex and $\Ss^3$ in the other
coordinates. Standard properties of polyhedral products give the
following results.

\begin{itemize}
\item Every $N_{\Hh}(\Lambda)$ is $2$-connected. If there are no
indispensable coordinates, it is $3$-connected and
\[
H^4(N_{\Hh}(\Lambda);\Zz)=\Zz e,
\]
where $e$ is the Euler class of the associated oriented real
rank-four bundle.
\item The natural embedding into $\Hh P^{n-1}$ has trivial normal
bundle of rank $4m$. With the orientation convention specified below,
\[
p_1(TN_{\Hh}(\Lambda))=2(n-2)e.
\]
\item No LVMQ manifold admits a quaternionic K\"ahler metric.
\end{itemize}

The last assertion uses the position of $N_{\Hh}$ in projective space.
It is the regular zero set of a map to $\Rr^{4m}$, so its Poincar\'e
dual in $\Hh P^{n-1}$ is zero. If there are no indispensable
coordinates and $\dim N_{\Hh}=4r$, this gives $e^r=0$. Since every
real degree-four class is a multiple of $e$, a quaternionic K\"ahler
fundamental form cannot exist. In the presence of an indispensable
coordinate, $H^4(N_{\Hh};\Rr)=0$ gives the same conclusion.
The obstruction also excludes a $3$-Sasakian structure on
$Z_{\Hh}$ whose Reeb action is the given diagonal $\Sp(1)$ action.

\subsection*{Real dilation orbits}

Let
\[
\cS_\Lambda=\{q\in\Hh^n:
0\in\conv\{\lambda_i:q_i\ne0\}\}.
\]
The action
\[
\Phi_T(q)_i=e^{\langle\lambda_i,T\rangle}q_i,
\qquad T\in\Rr^{4m},
\]
has a unique norm minimum on each orbit in $\cS_\Lambda$. The
normalized minima form $Z_{\Hh}(\Lambda)$, and
\[
\cS_\Lambda\cong Z_{\Hh}(\Lambda)\times\Rr^{4m}\times\Rr_{>0}.
\]
After dividing by left quaternionic scalar multiplication, this gives
\[
V_\Lambda:=\cS_\Lambda/\Hh^*\cong
N_{\Hh}(\Lambda)\times\Rr^{4m}.
\]
Thus $N_{\Hh}(\Lambda)$ is the Hausdorff leaf space of a real analytic
foliation on the open subset $V_\Lambda\subset\Hh P^{n-1}$.

In the complementary Poincar\'e domain, where
$0\notin\conv\Lambda$, a separating vector gives a one-parameter
subgroup whose orbits meet each centered sphere exactly once.
The remaining real action induces a singular trace foliation on the
sphere. On the stratum with support $I$, its leaves are
$\Rr^{\varrho(I)-1}$, where $\varrho(I)$ is the real rank of the corresponding
weights. This construction is related to the complex Poincar\'e
problem of Guckenheimer \cite{Guck} and to the higher-rank work of
Arroyo, Cabrera, Seade and Verjovsky.\footnote{A.~Arroyo, C.~Cabrera,
J.~Seade and A.~Verjovsky, \emph{Holomorphic linear $\Cc^k$-actions,
trace foliations, and higher-rank Poincar\'e dynamics}, preprint (2026),
\url{https://arxiv.org/abs/2608.01286}.}

\subsection*{Quaternionic geometry and rank-one foliations}

For every $m\ge1$ and $d\ge1$, an explicit configuration gives
\[
N_{\Hh}(\Lambda)\cong
\Ss^{4d+3}\times(\Ss^3)^{4m-1}.
\]
These manifolds admit homogeneous positive Einstein metrics.
The ambient Kraines form restricts to a closed $4$-form with trivial
contraction kernel. A coordinate subgroup of the appropriate
dimension acts in a generalized Hamiltonian fashion with trivial
principal isotropy. Consequently these examples are quaternionic
toric in the sense of Gentili--Gori--Sarfatti \cite{GGS}.

There is a separate integrability problem for the real four-plane
field $\mathcal D_X=X\Hh$ of a vector field on $\Hh^n$.
For $X(q)=Aq$, with $A\in GL(n,\Hh)$ acting on the left and $n\ge2$,
Frobenius integrability holds exactly when $A$ is a nonzero real
scalar matrix. In the nonlinear case, if $X$ is $C^2$,
$X(0)=0$, $DX(0)=\mu I$ with $\mu>0$, and $X\Hh$ is involutive off
the origin, then
\[
\mathcal D_X(q)=q\Hh
\]
on a sufficiently small punctured ball. Its spherical trace is
therefore the standard Hopf fibration. The proof uses the elementary
fact that quaternionic submanifolds of flat quaternionic space are
totally geodesic, followed by backward convergence of the flow of $X$
to the origin.

\subsection*{Relation with earlier constructions}

The polyhedral-product description belongs to the general theory of
$(D^{r+1},\Ss^r)^K$ \cite{BP,BBCG,BBCG2}. Quaternionic quotient
constructions were studied by Scott \cite{Scott} and Hopkinson
\cite{Hop}. Rigidity questions are treated by Gkeneralis and
Prassidis \cite{GP} and in a preprint of Gkeneralis;\footnote{I.~Gkeneralis,
\emph{Equivariant rigidity of complex and quaternionic moment--angle
manifolds}, preprint (2026), \url{https://arxiv.org/abs/2604.18455v2}.}
local quaternionic toric actions and their
characteristic data are studied by Batakidis and Gkeneralis \cite{BG}.
Foth \cite{Foth} and Gentili--Gori--Sarfatti \cite{GGS} develop
geometric structures defined by closed $4$-forms and group actions.

Here the diagonal quotient has a projective realization supplied by
the LVM norm-minimization argument. This realization gives its trivial
normal bundle and the obstruction to quaternionic K\"ahler metrics.
Coordinatewise conjugation identifies the left diagonal quotient
with a right diagonal quotient. Its cohomology can therefore be
studied through restriction to the diagonal subgroup of the
coordinatewise group. The orbit polytope of the full coordinatewise
action has dimension $n-4m-1$, whereas one quarter of the dimension
of $N_{\Hh}$ is $n-m-1$. In particular the induced action of
$\Sp(1)^n$ on $N_{\Hh}$ is not a quoric action in the sense of
Hopkinson \cite{Hop}.

\section{Admissible configurations and smoothness}

Throughout the nondegenerate theory we assume $m\ge1$.  The case $m=0$
reduces to ordinary quaternionic projective space and will not be considered.
Put $p=4m$ and identify $\Hh^m$ with $\Rr^p$.

\begin{definition}
A configuration
\[
\Lambda=(\lambda_1,\ldots,\lambda_n)\subset\Rr^p
\]
is \emph{admissible} if the following conditions hold:
\[
0\in\conv\{\lambda_1,\ldots,\lambda_n\},
\]
and
\[
0\in\conv\{\lambda_i:i\in I\}
\quad\Longrightarrow\quad
|I|>p.
\]
The first condition is the Siegel condition and the second is weak
hyperbolicity.
\end{definition}

Applying weak hyperbolicity to $I=[n]$, which is allowed by the Siegel
condition, gives $n>p=4m$ for every admissible configuration.

Admissibility implies that the origin is an interior point of
$\conv\Lambda$ in $\Rr^p$. Otherwise a supporting hyperplane through
$0$, or the affine span when it is proper, would place $0$ in the
convex hull of at most $p$ vectors by Carath\'eodory's theorem.
In particular the affine span is all of $\Rr^p$. Equivalently, the augmented vectors
\[
\widetilde\lambda_i=(\lambda_i,1)\in\Rr^{p+1}
\]
span $\Rr^{p+1}$.

Define
\[
Z_{\Hh}(\Lambda)=
\left\{q=(q_1,\ldots,q_n)\in\Hh^n:
\sum_i\lambda_i|q_i|^2=0,\quad
\sum_i|q_i|^2=1\right\}.
\]

\begin{theorem}\label{thm:smooth}
If $\Lambda$ is admissible, then
$Z_{\Hh}(\Lambda)$ is a compact smooth manifold of dimension
\[
4n-p-1=4n-4m-1.
\]
\end{theorem}

\begin{proof}
Consider
\[
F:\Hh^n\longrightarrow\Rr^p\times\Rr,
\qquad
F(q)=
\left(
\sum_i\lambda_i|q_i|^2,\,
\sum_i|q_i|^2
\right).
\]
Then
\[
Z_{\Hh}(\Lambda)=F^{-1}(0,1).
\]
For $q\in Z_{\Hh}(\Lambda)$ let
\[
I(q)=\{i:q_i\neq0\}.
\]
The differential is
\[
dF_q(h)
=
2\left(
\sum_i\lambda_i\operatorname{Re}(\overline q_i h_i),\,
\sum_i\operatorname{Re}(\overline q_i h_i)
\right).
\]
For $i\in I(q)$ the real number
$\operatorname{Re}(\overline q_i h_i)$ can be prescribed arbitrarily by
taking $h_i$ proportional to $q_i$.  Thus $dF_q$ is onto if and only if
the augmented vectors
\[
\{\widetilde\lambda_i:i\in I(q)\}
\]
span $\Rr^{p+1}$.

Since
\[
0=\sum_{i\in I(q)}|q_i|^2\lambda_i,\qquad
1=\sum_{i\in I(q)}|q_i|^2,
\]
the origin belongs to the convex hull of the vectors in the support.
If their affine span had dimension at most $p-1$, Carath\'eodory's theorem
in that affine span would place the origin in the convex hull of at most
$p$ of them.  This contradicts weak hyperbolicity.  Hence the augmented
vectors have rank $p+1$.

Thus $(0,1)$ is a regular value of $F$.  Compactness follows from the sphere
equation and the dimension is
\[
4n-(p+1).
\]
\end{proof}

\begin{corollary}\label{cor:support}
Every point of $Z_{\Hh}(\Lambda)$ has at least $4m+1$ nonzero coordinates.
\end{corollary}

\section{The orbit polytope}

A facet of a convex polytope is a face of codimension one.  A
$d$-dimensional polytope is called \emph{simple} if exactly $d$ facets
meet at every vertex.

Let
\[
P_\Lambda=
\left\{x=(x_1,\ldots,x_n)\in\Rr_{\geq0}^n:
\sum_i\lambda_i x_i=0,\quad
\sum_i x_i=1\right\}.
\]
Since $0$ is interior to $\conv\Lambda$, it can be written as a
convex combination with every coefficient positive: for each index
$i$, a small point on the ray opposite $\lambda_i$ belongs to the
convex hull, giving a representation of $0$ with positive coefficient
at $i$; average these representations over $i$. Thus $P_\Lambda$
contains a point with every coordinate positive. Its dimension is
\[
d=n-4m-1.
\]

The group $\Sp(1)^n$ acts on the right by
\[
(q_1,\ldots,q_n)\cdot(u_1,\ldots,u_n)
=
(q_1u_1,\ldots,q_nu_n).
\]

\begin{theorem}\label{thm:polytope}
The map
\[
\mu(q)=\bigl(|q_1|^2,\ldots,|q_n|^2\bigr)
\]
has image $P_\Lambda$, and its fibres are exactly the
$\Sp(1)^n$ orbits.  Hence
\[
Z_{\Hh}(\Lambda)/\Sp(1)^n\cong P_\Lambda.
\]
\end{theorem}

\begin{proof}
Only the statement about the fibres needs proof.  If $\mu(q)=\mu(q')$, then
$|q_i|=|q_i'|$ for every $i$.  For a nonzero coordinate there is a unit
quaternion $u_i$ such that $q_i'=q_i u_i$.  When $q_i=0$ the corresponding
factor is arbitrary.  Therefore $q$ and $q'$ lie in the same orbit.
\end{proof}

For $J\subset[n]$ set
\[
F_J=P_\Lambda\cap\{x_j=0:j\in J\}.
\]

\begin{proposition}\label{prop:faces}
The set $F_J$ is nonempty if and only if
$0\in\conv\{\lambda_i:i\notin J\}$. In that case $F_J$ contains a
point with $x_i>0$ for every $i\notin J$.
\end{proposition}

\begin{proof}
A point of $F_J$ gives a convex representation of $0$ by the weights
with indices outside $J$, and conversely. If $0$ lies in this convex
hull, the supporting-hyperplane argument above, using weak
hyperbolicity, puts $0$ in its interior. Averaging as above gives a
convex representation with all coefficients $x_i$, $i\notin J$,
strictly positive.
\end{proof}

A coordinate $i$ is called \emph{indispensable} if $x_i>0$ on all of
$P_\Lambda$.  Equivalently, the equation $x_i=0$ does not define a facet.

\begin{proposition}\label{prop:simple}
The polytope $P_\Lambda$ is simple.  Indispensable coordinates correspond
to inequalities $x_i\ge0$ which never become active; they do not create
additional facets.
\end{proposition}

\begin{proof}
Let $x$ be a vertex and let $I=\{i:x_i>0\}$ be its support.  If
$|I|>p+1$, the augmented columns $(\lambda_i,1)$, $i\in I$, are linearly
dependent.  A nonzero relation among them gives a nontrivial tangent
direction inside the coordinate face determined by the support of $x$,
contradicting that $x$ is a vertex.  Hence $|I|\le p+1$.  On the other hand, the defining equations give
$0\in\operatorname{conv}\{\lambda_i:i\in I\}$, so weak hyperbolicity gives
$|I|>p$.  Hence $|I|=p+1$.

The augmented columns $(\lambda_i,1)$, $i\in I$, are linearly independent:
otherwise the vertex would be degenerate, equivalently there would be a
nonzero tangent direction supported in $I$.

Let
\[
T=\{h\in\Rr^n:\textstyle\sum_i\lambda_ih_i=0,\ \sum_ih_i=0\}
\]
be the tangent space to the affine slice defining $P_\Lambda$; it does not
depend on the base point and has dimension $d=n-p-1$.  Independence of
$\{(\lambda_i,1):i\in I\}$ says precisely that $T\cap\Rr^I=0$, where
$\Rr^I\subset\Rr^n$ is the coordinate subspace supported on $I$.  Since
$\dim T=d=|[n]\setminus I|$, this forces the projection
$T\to\Rr^{[n]\setminus I}$, $h\mapsto(h_j)_{j\notin I}$, to be a linear
isomorphism, which is exactly the asserted independence of the restrictions
$x_j|_T$, $j\notin I$.  Thus the $d=n-p-1$
coordinate hyperplanes $x_j=0$, $j\notin I$, have independent restrictions
to the affine space defining $P_\Lambda$.  Each is therefore a facet near
$x$, and exactly $d$ facets meet at $x$.  Since $x$ was arbitrary,
$P_\Lambda$ is simple.  If $i$ is indispensable, then $x_i>0$ everywhere,
so $x_i=0$ is not among the facets.
\end{proof}

\section{The real Siegel action}

For later comparison with quaternionic vector fields, write the
infinitesimal generators of the real action explicitly.  For
$U\in\Rr^{4m}$ define
\[
X_U(q_1,\ldots,q_n)
=
\bigl(\langle\lambda_1,U\rangle q_1,\ldots,
\langle\lambda_n,U\rangle q_n\bigr).
\]

\begin{proposition}\label{prop:real-commuting}
For all $U,V\in\Rr^{4m}$,
\[
[X_U,X_V]=0.
\]
Consequently the distribution tangent to the canonical real dilation
orbits is involutive on every constant-rank stratum.
\end{proposition}

\begin{proof}
The coefficients $\langle\lambda_i,U\rangle$ and
$\langle\lambda_i,V\rangle$ are real constants.  On the $i$th
quaternionic coordinate both vector fields are therefore scalar radial
fields, and their linear endomorphisms commute.  Hence their Lie bracket
vanishes coordinate by coordinate.  Equivalently, these vector fields
are the infinitesimal generators of the abelian action
$\Rr^{4m}\curvearrowright\Hh^n$ used below.
\end{proof}

Define
\[
\cS_\Lambda=
\left\{
q\in\Hh^n:
0\in\conv\{\lambda_i:q_i\neq0\}
\right\}.
\]
For $T\in\Rr^p$ put
\[
\Phi_T(q_1,\ldots,q_n)
=
\left(
e^{\langle\lambda_1,T\rangle}q_1,\ldots,
e^{\langle\lambda_n,T\rangle}q_n
\right).
\]
All exponential factors are positive real numbers; hence
\[
\Phi_{T+S}=\Phi_T\circ\Phi_S.
\]
This defines an $\Rr^p$ action.

For fixed $q\in\cS_\Lambda$ define
\[
f_q(T)=\|\Phi_T(q)\|^2
=
\sum_i|q_i|^2e^{2\langle\lambda_i,T\rangle}.
\]

\begin{lemma}\label{lem:hessian}
One has
\[
\nabla f_q(T)=
2\sum_i|q_i|^2e^{2\langle\lambda_i,T\rangle}\lambda_i
\]
and
\[
D^2f_q(T)(X,X)
=
4\sum_i|q_i|^2e^{2\langle\lambda_i,T\rangle}
\langle\lambda_i,X\rangle^2.
\]
\end{lemma}

\begin{lemma}\label{lem:interior}
If $\Lambda$ is admissible and $q\in\cS_\Lambda$, then
\[
0\in\operatorname{int}
\conv\{\lambda_i:q_i\neq0\}
\]
in $\Rr^p$.
\end{lemma}

\begin{proof}
If the origin lay on the boundary, a supporting face through the origin would
be contained in an affine hyperplane of dimension at most $p-1$.
Carath\'eodory's theorem in this hyperplane would express the origin as a
convex combination of at most $p$ vectors, contradicting weak hyperbolicity.
\end{proof}

\begin{lemma}\label{lem:proper}
For $q\in\cS_\Lambda$, the function $f_q$ is strictly convex and proper.
\end{lemma}

\begin{proof}
Strict convexity follows from Lemma~\ref{lem:hessian}, because the support
spans $\Rr^p$.  For properness, Lemma~\ref{lem:interior} implies that in
every direction $X$ on the unit sphere some scalar product
$\langle\lambda_i,X\rangle$ is positive.  By compactness of the unit sphere,
the maximum of these positive scalar products is uniformly bounded away from
zero.  Hence along every sequence $|T|\to\infty$ at least one exponential
term tends to infinity.
\end{proof}

\begin{theorem}[Global transversal]\label{thm:transversal}
Every $\Rr^p$ orbit in $\cS_\Lambda$ contains a unique point of minimum
Euclidean norm.  At the minimum,
\[
\sum_i\lambda_i|q_i|^2=0.
\]
The map
\[
Z_{\Hh}(\Lambda)\times\Rr^p\times\Rr_{>0}
\longrightarrow \cS_\Lambda,
\qquad
(q,T,r)\longmapsto r\,\Phi_T(q),
\]
is a diffeomorphism.
\end{theorem}

\begin{proof}
Existence and uniqueness of the minimum follow from strict convexity and
properness.  By Lemma~\ref{lem:interior} the weights in the support of
$q$ span $\Rr^p$, so $\Phi_T(q)=q$ forces $\langle\lambda_i,T\rangle=0$ for all
$i$ in the support, hence $T=0$, and the action is free on
$\cS_\Lambda$;
hence $T$ is determined by the point $r\Phi_T(q)$.  By Lemma~\ref{lem:hessian}, the critical point equation is the
first defining equation of $Z_{\Hh}(\Lambda)$.  Normalization gives the
sphere equation.  Smooth dependence on the initial point follows from the
implicit function theorem, since the Hessian at the critical point is
positive definite.
\end{proof}

\section{Projectivization and LVMQ manifolds}

The diagonal left action
\[
a\cdot(q_1,\ldots,q_n)
=
(aq_1,\ldots,aq_n),
\qquad a\in\Sp(1),
\]
preserves $Z_{\Hh}(\Lambda)$ and is free.  Define
\[
N_{\Hh}(\Lambda)=Z_{\Hh}(\Lambda)/\Sp(1).
\]
Then
\[
\dim_{\Rr}N_{\Hh}(\Lambda)=4n-4m-4.
\]

We use the left-projective convention
\[
\Hh P^{n-1}
=
(\Hh^n\setminus\{0\})/\Hh^\ast
\]
for left scalar multiplication.  This is diffeomorphic to the usual
right-projective model.  Write
\[
\Hh^\ast=\Rr_{>0}\times\Sp(1).
\]
Then
\[
V_\Lambda=\cS_\Lambda/\Hh^\ast
\subset\Hh P^{n-1}.
\]

\begin{theorem}\label{thm:projective}
The $\Rr^{4m}$ action descends to $V_\Lambda$ and
\[
V_\Lambda\cong N_{\Hh}(\Lambda)\times\Rr^{4m}.
\]
In particular
\[
V_\Lambda/\Rr^{4m}\cong N_{\Hh}(\Lambda).
\]
\end{theorem}

\begin{proof}
The action by real positive coordinate dilations commutes with positive
homotheties and with diagonal left multiplication by $\Sp(1)$.
Apply Theorem~\ref{thm:transversal} and divide by these two actions.
\end{proof}

Thus $N_{\Hh}(\Lambda)$ is the Hausdorff leaf space of a real
analytic foliation of an open subset of quaternionic projective space,
in the same way as in the complex LVM construction.

\section{The coordinatewise action on the quotient}

The coordinatewise right $\Sp(1)^n$ action commutes with the diagonal left
$\Sp(1)$ and therefore descends to $N_{\Hh}(\Lambda)$.

\begin{proposition}\label{prop:isotropy}
Let $[q]\in N_{\Hh}(\Lambda)$ and
\[
J(q)=\{i:q_i=0\}.
\]
Then
\[
\operatorname{Stab}_{\Sp(1)^n}([q])
\cong
\Sp(1)\times\Sp(1)^{|J(q)|}.
\]
In particular the principal isotropy group is $\Sp(1)$.
\end{proposition}

\begin{proof}
The class $[q]$ is fixed by $u=(u_1,\ldots,u_n)$ precisely when there is
an $a\in\Sp(1)$ such that
\[
q_i u_i=a q_i
\]
for all $i$.  If $q_i\neq0$, then
\[
u_i=q_i^{-1}a q_i.
\]
If $q_i=0$, the element $u_i$ is arbitrary.
\end{proof}

If a point of $P_\Lambda$ lies in the relative interior of a codimension
$k$ face, exactly $k$ quaternionic coordinates vanish.  The isotropy therefore increases by
$k$ factors of $\Sp(1)$.

The orbit space of this action on $N_{\Hh}(\Lambda)$ is $P_\Lambda$.

\section{Relation with quaternionic moment-angle manifolds}

Let $P=P_\Lambda$.  We index the facet complex by all
coordinates, including indispensable ones.  Thus let $K_P$ be the
simplicial complex on $[n]$ defined by
\[
J\in K_P
\quad\Longleftrightarrow\quad
P\cap\{x_j=0:j\in J\}\neq\varnothing.
\]
If $i$ is indispensable, then $\{i\}$ is not a simplex; in the usual
polyhedral-product terminology $i$ is a ghost vertex.  When there are no
indispensable coordinates, $K_P$ is the nerve of the facets of the
simple polytope $P$.  For a simplicial complex $K$ on $[n]$ and a pair $(X,A)$, the
polyhedral product is
\[
(X,A)^K
=
\bigcup_{\sigma\in K}
\prod_{i=1}^n Y_i(\sigma),
\qquad
Y_i(\sigma)=
\begin{cases}
X,&i\in\sigma,\\
A,&i\notin\sigma.
\end{cases}
\]
The quaternionic moment-angle polyhedral product associated with $P$ is
\[
\cZ_P^{\Hh}=(D^4,\Ss^{3})^{K_P}.
\]
This is a special case of the general polyhedral-product formalism; see
\cite{BP,BBCG,BBCG2}.

\begin{theorem}\label{thm:polyprod}
There is an $\Sp(1)^n$-equivariant homeomorphism
\[
Z_{\Hh}(\Lambda)\cong(D^4,\Ss^{3})^{K_P},
\]
where an indispensable coordinate is represented by a ghost vertex of
$K_P$ and hence contributes an $\Ss^{3}$ factor.
\end{theorem}

\begin{proof}
By Theorem~\ref{thm:polytope}, the norm-square map is the orbit map and
there is a canonical model
\[
Z_{\Hh}(\Lambda)\cong(\Sp(1)^n\times P)/\!\sim,
\]
where $(u,x)\sim(u',x)$ precisely when $u_i=u_i'$ for every coordinate
with $x_i>0$; equivalently the $i$th $\Sp(1)$ factor is collapsed on the
facet $x_i=0$.  Since $P$ is simple by Proposition~\ref{prop:simple}, its
facet nerve is $K_P$.  The standard polyhedral-product model for a simple
polytope identifies
\[
(\Sp(1)^n\times P)/\!\sim
\quad\text{with}\quad
(D^4,\Ss^{3})^{K_P};
\]
see, for example, \cite{BP,BBCG}.  Concretely, polar coordinates
$z_i=\rho_i u_i$ identify the radial variables with a cubical
neighborhood of the face lattice of $P$, while the phase $u_i$ becomes
irrelevant exactly when $\rho_i=0$, which is exactly the above
 equivalence relation.  These local identifications agree on face
intersections and are $\Sp(1)^n$-equivariant. They leave the phase
coordinates unchanged and therefore also commute with the diagonal
left $\Sp(1)$ action.

If coordinate $i$ is indispensable, the facet $x_i=0$ is absent.  In the
simplicial description this is a ghost vertex, so the $i$th coordinate of
the polyhedral product is constrained to $\Ss^{3}$ everywhere and splits off
as an $\Ss^{3}$ factor.
\end{proof}

\begin{corollary}\label{cor:2connZ}
The manifold $Z_{\Hh}(\Lambda)$ is $2$-connected.
\end{corollary}

\begin{proof}
Use the CW structures with one $0$-cell and one $3$-cell on $\Ss^{3}$, and
with $\Ss^{3}$ as the boundary subcomplex of the $4$-cell $D^4$.  The induced
CW structure on $(D^4,\Ss^{3})^{K_P}$ has one $0$-cell and no cells of
dimension $1$ or $2$: every positive-dimensional cell has dimension at
least $3$.  By cellular approximation every map $\Ss^j\to Z_{\Hh}(\Lambda)$,
$j=1,2$, is homotopic into the $j$-skeleton, which is a point.  Thus
$\pi_1=\pi_2=0$.  Extra indispensable coordinates only add $\Ss^{3}$ factors
and do not change this conclusion.
\end{proof}

The additive (co)homology is given by the generalized Hochster decomposition.  For
$J\subset[n]$, let $(K_P)_J$ be the full subcomplex on $J$, with ghost
vertices retained when indispensable coordinates are present. We use
augmented reduced (co)homology, so the complex $\{\varnothing\}$
has $\widetilde H^{-1}\cong\Zz$.

\begin{theorem}[Generalized Hochster decomposition]\label{thm:hochster}
There are natural additive decompositions
\[
\widetilde H^q(Z_{\Hh}(\Lambda);\Zz)
\cong
\bigoplus_{\varnothing\ne J\subset[n]}
\widetilde H^{\,q-3|J|-1}((K_P)_J;\Zz)
\]
and
\[
\widetilde H_q(Z_{\Hh}(\Lambda);\Zz)
\cong
\bigoplus_{\varnothing\ne J\subset[n]}
\widetilde H_{\,q-3|J|-1}((K_P)_J;\Zz).
\]
\end{theorem}

\begin{proof}
This is the $(D^4,\Ss^{3})$ specialization of the standard stable splitting
and cellular decomposition for polyhedral products; see
\cite{BBCG,BP}.  In the stable splitting, a subset $J$ contributes a
suspension shift of $3|J|+1$, which gives the displayed degree
$q-3|J|-1$ after taking reduced (co)homology.  The construction is
integral and cellular, so it yields the homology and cohomology
statements simultaneously.
\end{proof}

\begin{proposition}\label{prop:stablepar}
The manifold $Z_{\Hh}(\Lambda)$ is stably parallelizable:
\[
TZ_{\Hh}(\Lambda)\oplus\Rr^{4m+1}
\cong\Rr^{4n}.
\]
\end{proposition}

\begin{proof}
The $4m+1$ gradients of the defining equations give a global framing of the
normal bundle of $Z_{\Hh}(\Lambda)$ in $\Hh^n\simeq\Rr^{4n}$.
\end{proof}

\section{Topology of the LVMQ quotient}

There is a principal bundle
\[
\Ss^{3}\longrightarrow Z_{\Hh}(\Lambda)
\stackrel{\pi}{\longrightarrow}N_{\Hh}(\Lambda).
\]

\begin{theorem}\label{thm:2connN}
Every LVMQ manifold is $2$-connected.
\end{theorem}

\begin{proof}
Use the homotopy exact sequence and
Corollary~\ref{cor:2connZ}, together with
\[
\pi_1(\Ss^{3})=\pi_2(\Ss^{3})=0.
\]
\end{proof}

Thus
\[
H^1(N_{\Hh};\Zz)=H^2(N_{\Hh};\Zz)=0.
\]

For the principal $\Sp(1)$ bundle
\[
\Sp(1)\longrightarrow Z_{\Hh}(\Lambda)\longrightarrow N_{\Hh}(\Lambda),
\]
we denote by
\[
e\in H^4(N_{\Hh}(\Lambda);\Zz)
\]
the Euler class of the associated real rank-four bundle
$E=Z_{\Hh}(\Lambda)\times_{\Sp(1)}\Hh$.
Here the left principal action is written as a right action
$q\cdot a=a^{-1}q$, and $\Sp(1)$ acts on the fiber $\Hh$ by left
multiplication. We orient $\Hh$ by $(1,i,j,k)$. If $E_{\Cc}$ is the
complex rank-two bundle obtained using right multiplication by $i$,
this is the opposite of its complex orientation, and
\[
e=-c_2(E_{\Cc}),\qquad p_1(E)=2e.
\]
This convention fixes the sign of $e$ throughout.

\begin{theorem}\label{thm:H4}
Assume that $\Lambda$ has no indispensable coordinates.  Then
$N_{\Hh}(\Lambda)$ is $3$-connected and
\[
H^4(N_{\Hh}(\Lambda);\Zz)\cong\Zz.
\]
The Euler class
\[
e\in H^4(N_{\Hh}(\Lambda);\Zz)
\]
of the principal $\Ss^{3}$ bundle is a generator.
\end{theorem}

\begin{proof}
By Theorem~\ref{thm:hochster}, in the absence of indispensable coordinates,
\[
H_3(Z_{\Hh};\Zz)=0,
\qquad
H^3(Z_{\Hh};\Zz)=H^4(Z_{\Hh};\Zz)=0.
\]
For $|J|=1$, the full subcomplex $K_J$ is a simplex, since there are
no ghost vertices. For $|J|\ge2$ and $q=3,4$, the degree
$q-3|J|-1$ in the generalized Hochster formula is less than $-1$.
Thus all the relevant summands vanish.

Since $Z_{\Hh}$ is $2$-connected, the Hurewicz homomorphism
$\pi_3(Z_{\Hh})\to H_3(Z_{\Hh};\Zz)$ is an isomorphism.  Hence
$\pi_3(Z_{\Hh})=0$.  The homotopy exact sequence of the principal $\Ss^{3}$
bundle then gives
\[
\pi_3(N_{\Hh})=0.
\]
Thus $N_{\Hh}$ is $3$-connected.

The Gysin sequence contains
\[
H^3(Z_{\Hh})
\longrightarrow H^0(N_{\Hh})
\stackrel{\smile e}{\longrightarrow}
H^4(N_{\Hh})
\longrightarrow H^4(Z_{\Hh}).
\]
The two outer groups vanish, so multiplication by $e$ identifies
$H^0(N_{\Hh})\cong\Zz$ with $H^4(N_{\Hh})$.
\end{proof}

In the absence of indispensable coordinates, this degree-four class is the
first nontrivial cohomology class of the LVMQ manifold.  

\subsection{Indispensable coordinates}

\begin{lemma}[Removing indispensable coordinates]
\label{lem:ghost-splitting}
Suppose that $k\ge1$ coordinates are indispensable.  Let
$Z_{\mathrm{red}}$ denote the quaternionic moment-angle space obtained
after removing the corresponding ghost vertices.  Then there are
homeomorphisms
\[
Z_{\Hh}(\Lambda)
\cong
Z_{\mathrm{red}}\times(\Ss^{3})^k
\]
and
\[
N_{\Hh}(\Lambda)
\cong
Z_{\mathrm{red}}\times(\Ss^{3})^{k-1}.
\]
\end{lemma}

\begin{proof}
An indispensable coordinate is a ghost vertex of the simplicial complex
and contributes an $\Ss^{3}$ factor to the polyhedral product.  Together
with Theorem~\ref{thm:polyprod}, this gives the first homeomorphism.

For the second, write a point in the product model as
\[
(z,g_1,\ldots,g_k).
\]
The diagonal left $\Sp(1)$ action can be gauged by the first factor:
\[
[z,g_1,\ldots,g_k]
\longmapsto
\bigl(g_1^{-1}z,\,
g_1^{-1}g_2,\ldots,g_1^{-1}g_k\bigr).
\]
This gives the stated homeomorphism of quotient spaces.
\end{proof}

\subsection{Dependence on the polytope}

\begin{theorem}\label{thm:combinatorial}
Let $\Lambda,\Lambda'\subset\Rr^{4m}$ be admissible configurations of
$n$ vectors. Suppose there are a permutation $\sigma$ of $[n]$ and a
homeomorphism $P_\Lambda\to P_{\Lambda'}$ which maps the face
$\{x_i=0\}$ onto the face $\{x'_{\sigma(i)}=0\}$ for every $i$.
Then
\[
N_{\Hh}(\Lambda)\cong N_{\Hh}(\Lambda')
\]
as topological manifolds.
\end{theorem}

\begin{proof}
After renumbering the coordinates of $\Lambda'$ by $\sigma$, the given
homeomorphism induces a homeomorphism between the two quotient
models
\[
(\Sp(1)^n\times P)/\sim.
\]
This homeomorphism leaves the quaternionic phase coordinates unchanged and
therefore commutes with the diagonal left $\Sp(1)$ action.  Passing to the
diagonal quotient gives the result.
\end{proof}

\section{Characteristic classes}

The principal bundle above is the restriction of the quaternionic Hopf
bundle
\[
\Ss^{3}\longrightarrow \Ss^{4n-1}\longrightarrow\Hh P^{n-1}.
\]
Let
\[
i:N_{\Hh}(\Lambda)\hookrightarrow\Hh P^{n-1}
\]
be the natural inclusion and let
\[
u\in H^4(\Hh P^{n-1};\Zz)
\]
be the Euler class of the real rank-four bundle associated with the
quaternionic Hopf bundle, with the orientation convention above.  Then
\[
e=i^\ast u.
\]

There is also a direct description of the normal bundle.  Define
\[
f:\Hh P^{n-1}\longrightarrow\Rr^{4m},
\qquad
f([q])=
\frac{\sum_i\lambda_i|q_i|^2}{\sum_i|q_i|^2}.
\]
Then
\[
N_{\Hh}(\Lambda)=f^{-1}(0).
\]

\begin{proposition}\label{prop:tangent}
The normal bundle of $N_{\Hh}(\Lambda)$ in $\Hh P^{n-1}$ is trivial and
\[
TN_{\Hh}(\Lambda)\oplus\Rr^{4m}
\cong i^\ast T\Hh P^{n-1}.
\]
\end{proposition}

\begin{proof}
The regularity argument of Theorem~\ref{thm:smooth}, after projectivization,
shows that $0$ is a regular value of $f$.  Since the target is the vector
space $\Rr^{4m}$, the differential provides a trivialization of the normal
bundle.
\end{proof}

There is also an equivariant calculation of the first Pontryagin
class. The framed normal bundle of $Z_{\Hh}$ in $\Hh^n$ descends to
a trivial bundle, and the tangent bundle along the principal fibers
descends to the adjoint bundle $\operatorname{ad}E$, associated with
conjugation on $\operatorname{Im}\Hh$. Hence
\[
TN_{\Hh}\oplus\operatorname{ad}E\oplus\Rr^{4m+1}
\cong E^{\oplus n}.
\]
For the complex rank-two representation of $\Sp(1)$ the formal Chern
roots are $x,-x$, so $e=x^2$. The complexified adjoint representation
has roots $2x,0,-2x$. The splitting principle therefore gives
\[
p_1(E)=2e,\qquad p_1(\operatorname{ad}E)=4e;
\]
see \cite{MS} for the characteristic-class conventions.

\begin{corollary}\label{cor:p1}
For every LVMQ manifold,
\[
p_1(TN_{\Hh}(\Lambda))=2(n-2)e.
\]
Moreover, $N_{\Hh}(\Lambda)$ is spin.
\end{corollary}

\begin{proof}
All bundles in the stable identity are oriented, so their
complexifications have first Chern class zero. The Whitney formula for
$c_2$ therefore gives integral additivity of $p_1=-c_2$ of the
complexification. Hence
$p_1(TN_{\Hh})=n(2e)-4e=2(n-2)e$.
Since $N_{\Hh}$ is $2$-connected,
$H^2(N_{\Hh};\Zz/2)=0$, so $w_2=0$.
\end{proof}

\begin{corollary}
If there are no indispensable coordinates, then
\[
\frac{p_1}{2}=(n-2)e\neq0.
\]
Consequently these LVMQ manifolds are not string manifolds; for a spin
manifold the string obstruction is the class $p_1/2\in H^4(-;\Zz)$.
\end{corollary}

\begin{proof}
By Theorem~\ref{thm:H4}, $e$ generates an infinite cyclic group, and
$n>4m\geq4$.
\end{proof}

\section{Explicit examples}

Since admissibility implies $n>4m$, the smallest possible value of $n$ is
$4m+1$.  We call this the minimal case.  In this case the configuration has
exactly $4m+1$ points in $\Rr^{4m}$; full affine rank therefore means that
these points are the vertices of a $4m$-simplex.

Assume
\[
n=4m+1
\]
and that the origin lies in the interior of this simplex.  Then
$P_\Lambda$ is a point.  The barycentric coordinates of the origin are
unique and positive, hence all radii $|q_i|$ are fixed.

\begin{proposition}\label{prop:minimal}
In this case
\[
Z_{\Hh}(\Lambda)\cong(\Ss^{3})^n,
\qquad
N_{\Hh}(\Lambda)\cong(\Ss^{3})^{n-1}=(\Ss^{3})^{4m}.
\]
\end{proposition}

\begin{proof}
Only the quaternionic phases vary in $Z_{\Hh}$.  Quotienting by the diagonal
left $\Ss^{3}$ action allows one to normalize one of these phases.
\end{proof}

In the minimal case all coordinates are indispensable and
\[
H^4(N_{\Hh};\Rr)=0.
\]

The equal-radius product of the standard round metrics on
\[
(\Ss^{3})^{4m}\cong \Sp(1)^{4m}
\]
is, up to an overall scale, the standard bi-invariant product metric on
the compact Lie group $\Sp(1)^{4m}$.  It is Einstein with positive scalar
curvature.

\subsection{Sphere products and Einstein metrics}

The following configuration gives the sphere-product examples.

\begin{lemma}[Simplex realization with indispensable coordinates]
\label{lem:simplex-realization}
Let $p=4m$ and $d\ge1$.  In $\Rr^p$ take
\[
\lambda_j=e_j\quad(1\le j\le p),
\qquad
\lambda_{p+\ell}=-(1,\ldots,1)
\quad(1\le\ell\le d+1).
\]
Then the configuration is admissible, the first $p$ coordinates are
precisely the indispensable ones, and its associated polytope is affinely
isomorphic to $\Delta^d$.
\end{lemma}

\begin{proof}
Write the polytope coordinates as
$(x_1,\ldots,x_p,y_1,\ldots,y_{d+1})$.  The weight equations are
\[
x_j=\sum_{\ell=1}^{d+1}y_\ell=:S
\qquad(1\le j\le p),
\]
and the normalization gives $(p+1)S=1$.  Hence
\[
x_j=\frac1{p+1},
\qquad
 y_\ell\ge0,
\qquad
\sum_{\ell=1}^{d+1}y_\ell=\frac1{p+1}.
\]
Thus $P_\Lambda$ is a scaled $d$-simplex, and the first $p$ coordinates
are positive everywhere.  Since $d\ge1$, omitting any one of the last
$d+1$ vectors still leaves at least one copy of $-(1,\ldots,1)$, so those
coordinates are not indispensable.

Finally, if the origin lies in the convex hull of a subconfiguration,
that subconfiguration must contain at least one negative vector.  If the
total coefficient of the negative vectors is $B>0$, the $j$th coordinate
equation forces the coefficient of every $e_j$ to equal $B$.  Hence all
$p$ positive basis vectors and at least one negative vector occur, so the
subconfiguration has at least $p+1$ elements.  This is weak hyperbolicity,
and the displayed convex relation also proves the Siegel condition.
\end{proof}

For this configuration the reduced quaternionic moment-angle manifold of
the $d$-simplex is
\[
(D^4,\Ss^{3})^{\partial\Delta^d}\cong \Ss^{4d+3}.
\]

\begin{theorem}[Sphere-product LVMQ manifolds]
\label{thm:einstein-family}
For every $m\ge1$ and every $d\ge1$ there are admissible configurations
for which
\[
N_{\Hh}(\Lambda)
\cong
\Ss^{4d+3}\times(\Ss^{3})^{4m-1}.
\]
\end{theorem}

\begin{proof}
Use the configuration of Lemma~\ref{lem:simplex-realization}, with
$p=4m$.  Its equations give
\[
|q_j|^2=\frac1{p+1}\qquad(1\le j\le p)
\]
and
\[
\sum_{\ell=1}^{d+1}|q_{p+\ell}|^2=\frac1{p+1}.
\]
Hence, after a constant rescaling of the factors,
\[
Z_{\Hh}(\Lambda)
\cong
\Ss^{4d+3}\times(\Ss^3)^p
\]
as smooth manifolds.  The diagonal left $\Sp(1)$ action is
\[
a\cdot(z,g_1,\ldots,g_p)
=
(az,ag_1,\ldots,ag_p).
\]
Since $p=4m\ge1$, use the first $\Ss^3$ factor as a global gauge:
\[
[z,g_1,\ldots,g_p]
\longmapsto
\bigl(g_1^{-1}z,\,
g_1^{-1}g_2,\ldots,g_1^{-1}g_p\bigr).
\]
This is a smooth diffeomorphism
\[
N_{\Hh}(\Lambda)
\cong
\Ss^{4d+3}\times(\Ss^3)^{p-1}
=
\Ss^{4d+3}\times(\Ss^3)^{4m-1}.
\]
\end{proof}

\begin{corollary}\label{cor:homogeneous-einstein-family}
Every manifold in Theorem~\ref{thm:einstein-family} admits a homogeneous
Einstein metric of positive scalar curvature.  Hence there are infinitely
many nonminimal LVMQ manifolds with positive Einstein metrics.
\end{corollary}

\begin{proof}
Give the sphere factors their standard round metrics.  If $\rho_1$ is
the radius of $\Ss^{4d+3}$ and $\rho_2$ is the common radius of the
$\Ss^3$ factors, then
\[
\operatorname{Ric}_{\Ss^{4d+3}(\rho_1)}
=
\frac{4d+2}{\rho_1^2}g_1,
\qquad
\operatorname{Ric}_{\Ss^3(\rho_2)}
=
\frac{2}{\rho_2^2}g_2.
\]
Choose
\[
\rho_1^2=(2d+1)\rho_2^2.
\]
All factors then have the same Einstein constant, so their Riemannian
product is Einstein.  The metric is homogeneous under
\[
SO(4d+4)\times SO(4)^{\,4m-1}.
\]
\end{proof}

\begin{remark}
The corollary concerns the standard homogeneous product metric on the
displayed product of spheres.  It does not assert that the metric obtained
by Riemannian submersion from the original quadratic realization is
Einstein.
\end{remark}

\section{Consequences for special geometry}

The Euclidean metric of $\Hh^n$ restricts to a natural invariant metric on
$Z_{\Hh}(\Lambda)$.  The free diagonal $\Ss^{3}$ action gives a quotient metric
on $N_{\Hh}(\Lambda)$.

\begin{theorem}\label{thm:nosymp}
A positive-dimensional compact LVMQ manifold admits no symplectic form.
Consequently it admits no K\"ahler or hyperk\"ahler metric.
\end{theorem}

\begin{proof}
By Theorem~\ref{thm:2connN},
\[
H^2(N_{\Hh};\Rr)=0.
\]
If $\omega$ were symplectic, then $[\omega]\neq0$, since a top power of
$\omega$ is a volume form on the compact manifold.  This is impossible.
\end{proof}

A Riemannian manifold of dimension $4r\ge8$ is quaternionic K\"ahler
if its holonomy is contained in $\Sp(r)\Sp(1)$. Such a metric has a
parallel fundamental $4$-form $\Omega$ for which $\Omega^r$ is a
nonzero multiple of its volume form \cite{Salamon}.

\subsection{No LVMQ manifold is quaternionic K\"ahler}

\begin{theorem}[Quaternionic K\"ahler obstruction]
\label{thm:no-QK}
No positive-dimensional LVMQ manifold admits a quaternionic K\"ahler
metric.
\end{theorem}

\begin{proof}
Write
\[
\dim N_{\Hh}(\Lambda)=4r.
\]
In our range $m\ge1$ and $n\ge4m+1$, so $r\ge3$; in particular we are
above real dimension four.

We separate two cases.

\smallskip
\noindent\emph{Case 1: there is at least one indispensable coordinate.}
Let $k\ge1$ be the number of indispensable coordinates.
By Lemma~\ref{lem:ghost-splitting},
\[
N_{\Hh}(\Lambda)
\cong
Z_{\mathrm{red}}\times(\Ss^{3})^{k-1},
\]
where $Z_{\mathrm{red}}$ has no indispensable coordinates.
The generalized Hochster formula gives
\[
H^4(Z_{\mathrm{red}};\Rr)=0,
\]
and the K\"unneth theorem therefore gives
\[
H^4(N_{\Hh}(\Lambda);\Rr)=0.
\]
A quaternionic K\"ahler manifold of dimension at least eight carries a
parallel fundamental $4$-form $\Omega$ whose top power is a nonzero
multiple of the Riemannian volume form \cite{Salamon}.  Hence
\[
[\Omega]\neq0\in H^4(N_{\Hh};\Rr),
\]
a contradiction.

\smallskip
\noindent\emph{Case 2: there are no indispensable coordinates.}
By Theorem~\ref{thm:H4},
\[
H^4(N_{\Hh};\Zz)=\Zz e.
\]
Let
\[
i:N_{\Hh}\hookrightarrow\Hh P^{n-1}
\]
be the natural embedding and let
$u\in H^4(\Hh P^{n-1};\Zz)$ be the Euler class of the Hopf bundle with the orientation chosen above,
so that
\[
e=i^\ast u.
\]

The manifold $N_{\Hh}$ is the regular zero set of the globally defined
map
\[
f:\Hh P^{n-1}\longrightarrow\Rr^{4m}.
\]
Equivalently, it is the zero set of a transverse section of the trivial
oriented bundle
\[
\Hh P^{n-1}\times\Rr^{4m}.
\]
Its Poincar\'e dual in $\Hh P^{n-1}$ is therefore the Euler class of a
trivial positive-rank bundle, hence zero:
\[
\operatorname{PD}_{\Hh P^{n-1}}[N_{\Hh}]=0.
\]
Consequently
\[
\left\langle e^r,[N_{\Hh}]\right\rangle
=
\left\langle
u^r\smile\operatorname{PD}[N_{\Hh}],
[\Hh P^{n-1}]
\right\rangle
=0.
\]
Since $H^{4r}(N_{\Hh};\Rr)$ is one-dimensional, this means
\[
e^r=0.
\]

If a quaternionic K\"ahler metric existed, its parallel fundamental
$4$-form would satisfy
\[
[\Omega]=c\,e
\]
for some real number $c$, because $H^4(N_{\Hh};\Rr)$ is
one-dimensional.  But then
\[
[\Omega]^r=c^r e^r=0,
\]
whereas $\Omega^r$ is a nonzero multiple of the volume form.  This is
again a contradiction.
\end{proof}

\begin{corollary}
The homogeneous Einstein metrics of
Corollary~\ref{cor:homogeneous-einstein-family} are not quaternionic
K\"ahler.
\end{corollary}

\subsection{Compatible 3-Sasakian structures}

A $3$-Sasakian structure is a Riemannian structure whose metric cone
is hyperk\"ahler. Its three Reeb vector fields generate a locally free
$\Sp(1)$ action. If this action is free, its quotient has a quaternionic
K\"ahler metric; see \cite[Chapter~13]{BoyerGalicki} and \cite{BGP}.

\begin{corollary}\label{cor:no-compatible-3sasakian}
The manifold $Z_{\Hh}(\Lambda)$ admits no $3$-Sasakian structure whose
Reeb $\Sp(1)$ action is the diagonal action defining
$N_{\Hh}(\Lambda)$.
\end{corollary}

\begin{proof}
Such a structure would induce a quaternionic K\"ahler metric on
$N_{\Hh}(\Lambda)$, contrary to Theorem~\ref{thm:no-QK}.
\end{proof}

\section{Commuting quaternionic coefficients}

Write
\[
\lambda_\alpha=
(\lambda_\alpha^1,\ldots,\lambda_\alpha^m).
\]
A direct quaternionic analogue of the complex linear LVM fields is obstructed
by noncommutativity.  There is a special commuting case.

Consider the real-linear fields defined by right multiplication
\[
V_r(q)_\alpha=q_\alpha\lambda_\alpha^r,
\qquad r=1,\ldots,m.
\]

\begin{proposition}\label{prop:commuting}
One has
\[
[V_r,V_s](q)_\alpha=
q_\alpha[\lambda_\alpha^r,\lambda_\alpha^s].
\]
Therefore the fields commute if and only if, for each fixed $\alpha$,
\[
[\lambda_\alpha^r,\lambda_\alpha^s]=0
\]
for all $r,s$.
\end{proposition}

For each $\alpha$, a family of commuting nonreal quaternions lies in a complex
slice
\[
\Cc_{I_\alpha}
=
\{a+bI_\alpha:a,b\in\Rr\}\subset\Hh.
\]
The slice is allowed to depend on $\alpha$.  In that case the full
configuration may still span $\Hh^m$ over $\Rr$.

For the fields defined by left multiplication,
\[
V_r(q)_\alpha=\lambda_\alpha^r q_\alpha,
\]
the same convention $[V,W]=DW(V)-DV(W)$ gives
\[
[V_r,V_s](q)_\alpha=
\bigl(\lambda_\alpha^s\lambda_\alpha^r-
\lambda_\alpha^r\lambda_\alpha^s\bigr)q_\alpha.
\]
Thus coordinatewise commuting coefficients give commuting fields in
both cases. The map $q\mapsto q\lambda$ is left-$\Hh$-linear, whereas
$q\mapsto\lambda q$ is right-$\Hh$-linear. Both are real-linear.
Commutation of the listed fields concerns their real span;
involutivity of their quaternionic span requires brackets involving
all quaternionic multiples.

\subsection{One fixed complex slice}

Assume there is a single unit imaginary quaternion $I$ such that all
coefficients lie in $\Cc_I$.  Then
\[
\Lambda\subset\Cc_I^m,
\]
so its real rank is at most $2m$.  This is therefore a rank-degenerate limit
of the full quaternionic situation.

Choose $J\perp I$ and write
\[
q_\alpha=z_\alpha+Jw_\alpha,
\qquad
z_\alpha,w_\alpha\in\Cc_I.
\]
Since
\[
|q_\alpha|^2=|z_\alpha|^2+|w_\alpha|^2,
\]
the defining equations become the ordinary complex equations for the
duplicated configuration
\[
\widetilde\Lambda=
(\lambda_1,\lambda_1,\ldots,\lambda_n,\lambda_n).
\]

\begin{proposition}\label{prop:fixedslice}
In the fixed-slice case,
\[
Z_{\Hh}(\Lambda)\cong Z_{\Cc}(\widetilde\Lambda).
\]
\end{proposition}

Thus the classical LVM intersections occur as a limiting case, in
which the rank and therefore the dimension of the intersection change.

\section{Wall-crossing}

The wall-crossing argument for intersections of quadrics is real convex
geometry.  It therefore applies to configurations in
\[
\Hh^m\simeq\Rr^{4m}.
\]

Let $\Lambda(t)$ be a one-parameter family which is admissible except at one
value $t_0$, where the origin crosses a single wall affinely generated by
$4m$ vectors.  Suppose the remaining vectors divide into $a$ on one side and
$b$ on the other.  Then
\[
a+b=n-4m=d+1,
\]
and $a,b\ge1$, since otherwise the Siegel condition fails on one side
of the wall.
A flip of type $(a,b)$ means that, in the simplicial complex dual to the
simple polytope, the subcomplex
\[
\Delta^{a-1}*\partial\Delta^{b-1}
\]
is replaced by
\[
\partial\Delta^{a-1}*\Delta^{b-1},
\]
where $*$ denotes the simplicial join and $\partial\Delta^0=\{\varnothing\}$.
For $b=1$ or $a=1$ the flip turns a ghost vertex into a vertex, or the
converse; this is how indispensable coordinates appear and disappear.

\begin{theorem}\label{thm:wall}
Across a generic wall-crossing of type $(a,b)$, the associated simple
polytope changes by a flip of type $(a,b)$.
\end{theorem}

\begin{proof}
This is the standard Gale-duality wall-crossing argument for intersections
of real quadrics; see Bosio--Meersseman \cite{BM}.  At a generic crossing
exactly one affine dependence changes sign.  In the Gale dual this replaces
the corresponding face by its complementary face, which is precisely a
flip of type $(a,b)$.  The argument uses only the oriented real
configuration in $\Rr^{4m}$ and therefore is unchanged by the use of
quaternionic coordinates.
\end{proof}

At the level of quaternionic moment-angle manifolds, the corresponding
topological replacement is expected to be
\[
\Ss^{4a-1}\times(\Ss^{3})^{4m}\times D^{4b}
\quad\longleftrightarrow\quad
D^{4a}\times(\Ss^{3})^{4m}\times \Ss^{4b-1}.
\]
The common boundary is
\[
\Ss^{4a-1}\times(\Ss^{3})^{4m}\times \Ss^{4b-1}.
\]

A proof requires invariant tubular neighborhoods and an explicit gluing
map; see Problem~\ref{prob:wall}.

\section{Relation with quaternionic toric geometry}

There are two different quaternionic toric frameworks relevant here.

Hopkinson introduced quoric manifolds as quaternionic analogues of
quasitoric manifolds: smooth manifolds with locally standard
$(\Ss^{3})^d$ actions and simple-polytope orbit spaces \cite{Hop}; compare
also Scott's quaternionic toric varieties
\cite{Scott}.  Batakidis and Gkeneralis classify local quaternionic
toric actions by a characteristic pair and a degree-four Euler class
\cite{BG}; compare the class $e$ of Theorem~\ref{thm:H4}.

Gentili, Gori and Sarfatti developed a differential-geometric theory based on
$4$-plectic manifolds and generalized Hamiltonian $\Sp(1)^d$ actions
\cite{GGS}; see also Foth \cite{Foth}.  A $4$-plectic form is a closed $4$-form $\psi$ for which
\[
v\longmapsto \iota_v\psi
\]
has trivial kernel.  An action preserving $\psi$ is \emph{generalized Hamiltonian} if
$\iota_Y\psi$ is an exact $3$-form for every fundamental vector field
$Y$. A \emph{tri-moment map} has one real component for each
$\Sp(1)$ factor; the differential of that component is the contraction
of $\psi$ by the oriented triple of fundamental fields of the factor.
The map is required to be invariant under the group. Its target is
\[
\left(\bigoplus_{\alpha=1}^r\Lambda^3\mathfrak{sp}(1)\right)^*
\cong\Rr^r,
\]
as in \cite[Definition~2.3]{GGS}. Thus the defining contractions use
the three generators of each individual factor. By
\cite[Definition~1.1]{GGS}, a compact connected $4r$-manifold is quaternionic toric
if it has a $4$-plectic form and an effective generalized Hamiltonian
$\Sp(1)^r$ action with discrete principal isotropy.

These theories involve structures not present automatically on a general
LVMQ quotient.  For the explicit sphere-product family, however, the
ambient Kraines form supplies the required $4$-plectic structure.

\subsection{A quaternionic toric sphere-product family}

On $\Hh^s$ with the Euclidean inner product put
$\omega_\alpha(v,w)=\langle I_\alpha v,w\rangle$, where $I_1,I_2,I_3$
are left multiplication by $i,j,k$, and let
\[
\Omega_{\mathrm K}=\omega_1\wedge\omega_1+\omega_2\wedge\omega_2+
\omega_3\wedge\omega_3
\]
be the Kraines $4$-form. Left multiplication is the convention that
descends, through the horizontal spaces of the left Hopf map, to the
quaternionic structure of the left-projective space $\Hh P^{n-1}$; the
pullback of the Kraines form of $\Hh P^{n-1}$ to a horizontal space is
$\Omega_{\mathrm K}$ up to a positive constant.

\begin{lemma}\label{lem:kraines-imaginary}
For $s\ge2$, the restriction of $\Omega_{\mathrm K}$ to
\[
(\operatorname{Im}\Hh)^s\subset\Hh^s
\]
is nondegenerate in the $4$-plectic sense.  More generally, for every
$t\ge0$ its restriction to
\[
(\operatorname{Im}\Hh)^s\oplus\Hh^t
\]
is nondegenerate.
\end{lemma}

\begin{proof}
On $\operatorname{Im}\Hh\simeq\Rr^3$, the three numbers
\[
\bigl(\omega_1(x,y),\omega_2(x,y),\omega_3(x,y)\bigr)
\]
are, up to a fixed choice of orientation and signs, the coordinates of
$x\times y$.

Let $0\ne x=(x_1,\ldots,x_s)\in(\operatorname{Im}\Hh)^s$, and after
reordering suppose that $x_1\ne0$.  Choose $y$ supported in the first
summand so that $x_1\times y_1\ne0$.  Since $s\ge2$, choose $z,w$
supported in the second summand so that $z_2\times w_2$ is a nonzero
multiple of $x_1\times y_1$.  The mixed two-form terms vanish because
the two pairs have disjoint supports, and therefore
\[
\Omega_{\mathrm K}(x,y,z,w)
=
2\langle x_1\times y_1,z_2\times w_2\rangle\ne0.
\]
Thus $\iota_x\Omega_{\mathrm K}\ne0$.

For $(\operatorname{Im}\Hh)^s\oplus\Hh^t$, if the imaginary component
of a nonzero vector is nonzero, use the preceding argument with the
other three vectors in $(\operatorname{Im}\Hh)^s$.  If its imaginary
component is zero, use the nondegeneracy of the standard Kraines form
on the quaternionic summand $\Hh^t$.
\end{proof}

\begin{theorem}[Quaternionic toric LVMQ sphere products]
\label{thm:lvmq-quaternionic-toric}
For every $m\ge1$ and $d\ge1$, the LVMQ manifolds of
Theorem~\ref{thm:einstein-family},
\[
N_{\Hh}(\Lambda)
\cong
\Ss^{4d+3}\times(\Ss^3)^{4m-1},
\]
are quaternionic toric manifolds in the sense of
Gentili--Gori--Sarfatti \cite{GGS}.  A $4$-plectic form is given by the
restriction of the Kraines form of $\Hh P^{n-1}$.

The minimal examples
\[
N_{\Hh}(\Lambda)\cong(\Ss^3)^{4m}
\]
have the same property.
\end{theorem}

\begin{proof}
Use the explicit configuration of Lemma~\ref{lem:simplex-realization}
and put $p=4m$, so that $n=p+d+1$.  At a convenient point of
$Z_{\Hh}(\Lambda)$, the first $p$ coordinates are positive real numbers
of fixed modulus, while in the last block only one coordinate is
nonzero.  The first $p$ fixed-radius factors contribute
$(\operatorname{Im}\Hh)^p$, and the tangent space to the sphere in the
last block contributes
\[
\operatorname{Im}\Hh\oplus\Hh^d.
\]
The vertical tangent space of the diagonal left $\Sp(1)$ action is one
weighted diagonal copy of $\operatorname{Im}\Hh$.  A real orthogonal
change among these $p+1$ quaternionic coordinate slots sends that
weighted diagonal to the last coordinate axis; because the coefficients
are real, it is quaternionic linear and preserves $\Omega_{\mathrm K}$.
Hence the horizontal tangent space of the quotient is equivalent to
\[
(\operatorname{Im}\Hh)^p\oplus\Hh^d
=
(\operatorname{Im}\Hh)^{4m}\oplus\Hh^d.
\]
By Lemma~\ref{lem:kraines-imaginary}, the restriction of the Kraines
form has trivial contraction kernel there.

Let $\Sp(1)^p$ act by right multiplication on the first $p$
coordinates, and let $\Sp(d+1)$ act on the last $d+1$ coordinates,
written as a row vector, by $q\mapsto qB$ with $BB^\ast=I$. The group
\[
\Sp(1)^p\times\Sp(d+1)
\]
acts transitively on the corresponding product model of
$Z_{\Hh}(\Lambda)$, preserves the defining equations, commutes with the
diagonal left $\Sp(1)$ action, and preserves the ambient quaternionic
projective metric and its Kraines form.  The preceding calculation
therefore holds at every point of $N_{\Hh}(\Lambda)$.  Since the
Kraines form is closed, its restriction is $4$-plectic.

Put
\[
r=\frac14\dim N_{\Hh}(\Lambda)=d+3m.
\]
Choose $r$ coordinate factors of the right $\Sp(1)^n$ action and denote
their product by $G\cong\Sp(1)^r$.  There remain
\[
n-r=m+1
\]
unselected coordinates.  On the open dense stratum where all
coordinates are nonzero, suppose that $u\in G$ fixes $[q]$.  Then there
is $a\in\Sp(1)$ such that
\[
a q_i=q_i u_i
\]
on selected coordinates and $a q_j=q_j$ on unselected coordinates.
Since an unselected coordinate is nonzero, $a=1$, and then all
$u_i=1$.  Thus the principal isotropy is trivial, and the action is
effective.

The restricted Kraines form $\Omega$ is invariant under $G$.
The sphere-product description gives $H^4(N_{\Hh};\Rr)=0$, so
$\Omega=d\beta$ for a $3$-form $\beta$. Averaging $\beta$ over the
compact group $G$ gives an invariant primitive, still denoted by
$\beta$. For every fundamental field $Y$, Cartan's formula gives
\[
\iota_Y\Omega=\iota_Yd\beta=-d(\iota_Y\beta).
\]
Thus the action is generalized Hamiltonian in the sense of
\cite[Definition~2.2]{GGS}.

For an
$\Sp(1)$ factor with infinitesimal generators $X_1,X_2,X_3$, the
one-form
\[
\iota_{X_3}\iota_{X_2}\iota_{X_1}\Omega
\]
is closed.  Indeed, the $3$-vector
$X_1\wedge X_2\wedge X_3$ has zero Lie algebra boundary:
$[X_1,X_2]\wedge X_3-[X_1,X_3]\wedge X_2+
[X_2,X_3]\wedge X_1=0$. Each term vanishes because each bracket
is a multiple of the remaining generator. Cartan's formula, together
with invariance and $d\Omega=0$, then gives closedness.  Since every LVMQ manifold
is $2$-connected, $H^1(N_{\Hh};\Rr)=0$.  Hence these closed one-forms
are exact, and their primitives give the tri-moment map.  Averaging these primitives over $G$ makes the tri-moment map
invariant; this is also \cite[Proposition~2.4]{GGS}.
Therefore $N_{\Hh}(\Lambda)$ is quaternionic toric.

In the minimal case the same tangent-space calculation leaves
$(\operatorname{Im}\Hh)^{4m}$ after quotienting the diagonal
$\operatorname{Im}\Hh$.  Lemma~\ref{lem:kraines-imaginary} again gives
$4$-plectic nondegeneracy, and the same symmetry and tri-moment
argument applies.
\end{proof}

\begin{remark}
The polytope $P_\Lambda$ is not the tri-moment polytope of this
quaternionic toric action.  Indeed,
\[
\dim P_\Lambda=n-4m-1,
\qquad
\frac14\dim N_{\Hh}(\Lambda)=n-m-1,
\]
and these dimensions differ for $m\ge1$.  Thus the LVM orbit polytope
and the polyhedral data of the quaternionic toric structure play
different roles.
\end{remark}

\subsection{\texorpdfstring{Subgroups of $\Sp(1)^n$}{Subgroups of Sp(1) to the nth power}}

For ordinary tori, an integer matrix determines a homomorphism
\[
(\Ss^{1})^r\longrightarrow(\Ss^{1})^n.
\]
There is no analogous statement for products of $\Sp(1)$.  At Lie algebra
level, $\mathfrak{sp}(1)$ is simple, and several source factors cannot be
combined arbitrarily into one target factor.

Consequently the real Gale matrix attached to $P_\Lambda$ does not, by
itself, define a subgroup
\[
\Sp(1)^r\subset\Sp(1)^n.
\]

\subsection{The degree-four class}

Degree-four Euler classes also appear in the classification of local
quaternionic toric actions \cite{BG}.  A comparison with quoric or $4$-plectic quotients therefore requires
additional characteristic data attached to the facets of $P_\Lambda$;
the real configuration $\Lambda$ alone does not define a quaternionic
lattice subgroup.

\begin{problem}
Determine which general admissible configurations $\Lambda$, beyond the
sphere-product family of Theorem~\ref{thm:lvmq-quaternionic-toric}, admit
a quaternionic toric structure compatible with a subgroup of the
coordinatewise $\Sp(1)^n$ action, and describe the corresponding
tri-moment data.
\end{problem}

\begin{problem}
Relate the Euler class
\[
e=i^\ast u
\]
of the diagonal Hopf bundle to the Euler class occurring in the
classification of local quaternionic toric actions.
\end{problem}

\section{The real Poincar\'e domain and trace foliations on the sphere}
\label{sec:poincare-quaternionic}

Assume in this section that
\[
0\notin\conv\{\lambda_1,\ldots,\lambda_n\}.
\]
The positive real dilation action then has a conical description,
with a singular trace foliation on every centered sphere. This is the
real-action counterpart of the complex trace construction of
Arroyo, Cabrera, Seade and Verjovsky mentioned above; the rank-one
complex problem goes back to
Guckenheimer \cite{Guck}. The quaternionic four-plane field $X\Hh$
is considered separately below.

\subsection{The real action and its coordinate stratification}

Put
\[
p=4m
\]
and consider, for $T\in\Rr^p$, the action
\begin{equation}\label{eq:poincare-action}
\Phi_T(q_1,\ldots,q_n)
=
\left(
e^{\langle\lambda_1,T\rangle}q_1,\ldots,
e^{\langle\lambda_n,T\rangle}q_n
\right).
\end{equation}
The factors are positive real numbers, so this action preserves the
quaternionic direction of every nonzero coordinate.

For $q\in\Hh^n$ let
\[
I(q)=\{j:q_j\neq0\}
\]
be its support.  For $I\subset[n]$ define
\[
E_I=
\{q\in\Hh^n:q_j=0\ \hbox{for }j\notin I\},
\qquad
E_I^\ast=
\{q\in E_I:q_j\neq0\ \hbox{for }j\in I\}.
\]
Then
\[
\Hh^n=\bigsqcup_{I\subset[n]}E_I^\ast,
\]
and each stratum is invariant under~\eqref{eq:poincare-action}.

For a subset $I$ put
\[
\varrho(I)=\rank_{\Rr}\{\lambda_i:i\in I\},
\qquad
K_I=
\{T\in\Rr^p:
\langle\lambda_i,T\rangle=0\text{ for every }i\in I\}.
\]

\begin{proposition}\label{prop:poincare-orbits}
Let $q\in E_I^\ast$.  The isotropy subgroup of $q$ for the action
\eqref{eq:poincare-action} is the vector subspace $K_I$.  Consequently
the ambient orbit through $q$ is naturally diffeomorphic to
\[
\Rr^p/K_I\cong\Rr^{\varrho(I)}.
\]
In particular its dimension depends only on the support.
\end{proposition}

\begin{proof}
Since $q_i\neq0$ for $i\in I$, the equality
$\Phi_T(q)=q$ is equivalent to
\[
e^{\langle\lambda_i,T\rangle}=1
\qquad(i\in I).
\]
The exponent is real, so this is equivalent to
\[
\langle\lambda_i,T\rangle=0
\qquad(i\in I).
\]
Thus the isotropy is exactly $K_I$.  The remaining statements follow
from the rank theorem.
\end{proof}

\begin{remark}\label{rem:no-arithmetic}
Since the exponents are real, the isotropy $K_I$ is a vector subspace
and has no discrete part. Every ambient leaf through a nonzero point is
diffeomorphic to a Euclidean space.
\end{remark}

The orbit decomposition is a Stefan--Sussmann singular foliation: it is
the partition into immersed orbits of the family of commuting complete
vector fields
\[
X_T(q)
=
\left(
\langle\lambda_1,T\rangle q_1,\ldots,
\langle\lambda_n,T\rangle q_n
\right),
\qquad T\in\Rr^p.
\]
On every coordinate stratum its rank is constant.

\subsection{The real convex Poincar\'e condition}

\begin{definition}
The quaternionic configuration $\Lambda$ lies in the
\emph{Poincar\'e domain} if
\[
0\notin\conv\{\lambda_1,\ldots,\lambda_n\}
\subset\Rr^{4m}.
\]
\end{definition}

By strict separation of a compact convex set from the origin, the
Poincar\'e condition is equivalent to the existence of a vector
$\xi\in\Rr^p$ such that
\begin{equation}\label{eq:positive-xi}
\langle\lambda_j,\xi\rangle>0,
\qquad j=1,\ldots,n.
\end{equation}
Denote the open cone of all such vectors by
\[
C_\Lambda
=
\{\xi\in\Rr^p:
\langle\lambda_j,\xi\rangle>0\text{ for every }j\}.
\]

For $r>0$ write
\[
\Ss_r^{4n-1}
=
\{q\in\Hh^n:\|q\|=r\}.
\]

\begin{theorem}[Conical model]\label{thm:q-cone}
Assume that $\Lambda$ lies in the Poincar\'e domain and fix
$\xi\in C_\Lambda$.  Then
\[
\Theta_r:\Ss_r^{4n-1}\times\Rr
\longrightarrow
\Hh^n\setminus\{0\},
\qquad
(q,t)\longmapsto\Phi_{t\xi}(q),
\]
is a diffeomorphism.

Moreover, if $\mathcal F_\Lambda$ denotes the ambient orbit foliation,
then $\Theta_r$ identifies every punctured ambient leaf with the
product of its intersection with $\Ss_r^{4n-1}$ and $\Rr$.  Thus
$\mathcal F_\Lambda$ on $\Hh^n\setminus\{0\}$ is the open cone over its
trace foliation on the sphere.
\end{theorem}

\begin{proof}
For $q\neq0$,
\[
g_q(t)=\|\Phi_{t\xi}(q)\|^2
=
\sum_j|q_j|^2
e^{2t\langle\lambda_j,\xi\rangle}.
\]
By~\eqref{eq:positive-xi},
\[
g_q'(t)
=
2\sum_j
|q_j|^2
e^{2t\langle\lambda_j,\xi\rangle}
\langle\lambda_j,\xi\rangle
>0.
\]
Also
\[
\lim_{t\to-\infty}g_q(t)=0,
\qquad
\lim_{t\to+\infty}g_q(t)=+\infty.
\]
Hence every $\Rr\xi$ orbit meets $\Ss_r^{4n-1}$ exactly once.
The hitting time depends smoothly on $q$ by the implicit function
theorem, because $g_q'(t)$ never vanishes.  This gives the
diffeomorphism and its inverse.

Since $\Rr\xi$ is contained in the acting group, the same
decomposition applies leaf by leaf.
\end{proof}

\subsection{The trace action}

Fix $\xi\in C_\Lambda$ and choose a real hyperplane
\[
H_\xi\subset\Rr^p,
\qquad
\Rr^p=H_\xi\oplus\Rr\xi.
\]
For $U\in H_\xi$ and $q\in \Ss_r^{4n-1}$, there is a unique real number
$\tau(U,q)$ such that
\[
\Phi_{U+\tau(U,q)\xi}(q)\in \Ss_r^{4n-1}.
\]
Define
\begin{equation}\label{eq:trace-action}
\Psi_U(q)
=
\Phi_{U+\tau(U,q)\xi}(q).
\end{equation}

\begin{theorem}[Trace action]\label{thm:q-trace-action}
The maps $\Psi_U$, $U\in H_\xi$, define a smooth action of the
additive vector group $H_\xi$ on $\Ss_r^{4n-1}$.  Its orbits are exactly
the connected components of the intersections of the nonzero ambient
orbits with the sphere.

If $q\in E_I^\ast\cap \Ss_r^{4n-1}$, then the trace leaf through $q$ has
dimension
\[
\varrho(I)-1
\]
and is diffeomorphic to
\[
\Rr^{\varrho(I)-1}.
\]
\end{theorem}

\begin{proof}
Existence and smoothness of $\tau$ follow from the same strict
monotonicity used in Theorem~\ref{thm:q-cone}.  The uniqueness of the
radial correction gives
\[
\Psi_U\circ\Psi_V=\Psi_{U+V},
\]
so this is an action.

The ambient orbit through $q$ is $\Rr^{\varrho(I)}$ by
Proposition~\ref{prop:poincare-orbits}.  The vector field generated by
$\xi$ is everywhere transverse to the sphere, because
\[
\frac{d}{dt}\bigg|_{t=0}
\|\Phi_{t\xi}(q)\|^2
=
2\sum_i|q_i|^2\langle\lambda_i,\xi\rangle>0.
\]
Therefore the intersection with the sphere has codimension one inside
the ambient orbit and dimension $\varrho(I)-1$.

The trace isotropy is a vector subspace of $H_\xi$.  Hence the trace
leaf is a quotient of a real vector space by a vector subspace, and is
therefore diffeomorphic to $\Rr^{\varrho(I)-1}$.
\end{proof}

The infinitesimal generators of the trace action can be written
explicitly.  For $U\in H_\xi$, set
\[
a_U(q)
=
\frac{
\sum_i|q_i|^2\langle\lambda_i,U\rangle
}{
\sum_i|q_i|^2\langle\lambda_i,\xi\rangle
}.
\]
Then the tangent field to the trace action is
\begin{equation}\label{eq:trace-vector-field}
Y_U(q)
=
X_U(q)-a_U(q)X_\xi(q).
\end{equation}
Indeed,
\[
d(\|q\|^2)_q(Y_U(q))=0.
\]
The denominator in $a_U$ is everywhere positive on the sphere.

\begin{corollary}\label{cor:qtrace-singular}
The trace decomposition on $\Ss_r^{4n-1}$ is a
Stefan--Sussmann singular foliation.  On the stratum
$E_I^\ast\cap \Ss_r^{4n-1}$ it is regular of dimension $\varrho(I)-1$.
\end{corollary}

\subsection{The generic rank case}

Assume now that every subconfiguration has maximal possible real rank,
\begin{equation}\label{eq:qgeneric}
\varrho(I)=\min\{|I|,p\},
\qquad I\subset[n].
\end{equation}
\begin{corollary}\label{cor:qgeneric-trace}
Under~\eqref{eq:qgeneric}, if $|I|=s$, then every trace leaf in
$E_I^\ast\cap \Ss_r^{4n-1}$ is diffeomorphic to
\[
\begin{cases}
\Rr^{s-1},&s\le p,\\[2mm]
\Rr^{p-1},&s\ge p.
\end{cases}
\]
In particular, if $n\ge4m$, the principal trace leaves are all
$\Rr^{4m-1}$.
\end{corollary}

\subsection{The leaf spaces on the coordinate strata}

Fix a support $I$, $|I|=s$.  Every point of $E_I^\ast$ can be written
uniquely as
\[
q_i=e^{\rho_i}u_i,
\qquad
\rho_i\in\Rr,\quad
u_i\in \Ss^{3},
\qquad i\in I.
\]
Hence
\[
E_I^\ast\cong\Rr^s\times(\Ss^{3})^s.
\]
Let
\[
A_I:\Rr^p\longrightarrow\Rr^s,
\qquad
A_I(T)
=
\bigl(
\langle\lambda_i,T\rangle
\bigr)_{i\in I}.
\]
In logarithmic radial coordinates, the action is simply
\[
(\rho,u)\longmapsto(\rho+A_I(T),u).
\]

\begin{theorem}[Stratum leaf space]\label{thm:q-stratum-quotient}
For every support $I$, the ambient orbit space on the coordinate
stratum is naturally a smooth manifold
\[
E_I^\ast/\Rr^p
\cong
\Rr^s/A_I(\Rr^p)\times(\Ss^{3})^s
\cong
\Rr^{s-\varrho(I)}\times(\Ss^{3})^s.
\]
By the conical model, the same space is the leaf space of the trace
foliation on $E_I^\ast\cap \Ss_r^{4n-1}$:
\[
\bigl(E_I^\ast\cap \Ss_r^{4n-1}\bigr)/
\widehat{\mathcal F}_\Lambda
\cong
\Rr^{s-\varrho(I)}\times(\Ss^{3})^s.
\]
\end{theorem}

\begin{proof}
The first statement follows from the logarithmic radial coordinates:
the action translates the $\Rr^s$ factor by the closed vector subspace
$A_I(\Rr^p)$ and leaves all quaternionic phases fixed.  The quotient is
therefore
\[
\Rr^s/A_I(\Rr^p)\times(\Ss^{3})^s.
\]
The second statement follows because
Theorem~\ref{thm:q-cone} identifies every punctured ambient leaf with a
trace leaf times $\Rr$ and therefore identifies the two leaf spaces.
\end{proof}

On the principal stratum, assuming full rank,
\[
\left(\Ss_r^{4n-1}\cap(\Hh^\ast)^n\right)/
\widehat{\mathcal F}_\Lambda
\cong
\Rr^{n-4m}\times(\Ss^{3})^n.
\]
This is an explicit smooth leaf space for the regular part of the
trace foliation.

\subsection{Orbit closures}

\begin{proposition}[Support monotonicity]\label{prop:q-support}
Let $q\in \Ss_r^{4n-1}$ and let $y$ belong to the closure of its trace
leaf.  Then
\[
I(y)\subseteq I(q).
\]
\end{proposition}

\begin{proof}
If $q_j=0$, then the $j$th coordinate remains zero under the ambient
action and under the radial correction.  The same is true on the
closure.
\end{proof}

\begin{proposition}[No same-support recurrence]\label{prop:no-recurrence}
A trace leaf is closed inside its support stratum
\[
E_I^\ast\cap \Ss_r^{4n-1}.
\]
Consequently every limit point of a trace leaf which is not on the leaf
itself has strictly smaller support.
\end{proposition}

\begin{proof}
In logarithmic radial coordinates on $E_I^\ast$, an ambient orbit is
an affine translate of the closed vector subspace $A_I(\Rr^p)$,
together with a fixed point of $(\Ss^{3})^s$.  Hence it is closed in
$E_I^\ast$.  Its intersection with the sphere is therefore closed in
$E_I^\ast\cap \Ss_r^{4n-1}$.  The second assertion follows from
Proposition~\ref{prop:q-support}.
\end{proof}

Quasiperiodic winding does not occur here because the action does not
rotate the quaternionic phases.

\subsection{Commutation with the quaternionic phase action}

The right coordinatewise action of $\Sp(1)^n$,
\[
(q_1,\ldots,q_n)\cdot
(u_1,\ldots,u_n)
=
(q_1u_1,\ldots,q_nu_n),
\]
commutes with the real action~\eqref{eq:poincare-action}.  It also
commutes with the trace action~\eqref{eq:trace-action}, since the
radial correction depends only on the norms $|q_i|$.

\begin{proposition}\label{prop:q-phase-transverse}
The group $\Sp(1)^n$ acts by automorphisms of the trace foliation.
On each support stratum, the quotient description of
Theorem~\ref{thm:q-stratum-quotient} is equivariant with respect to
the natural $\Sp(1)^I$ action on the factor $(\Ss^{3})^I$.
\end{proposition}

\subsection{The standard 3-Sasakian structure}

A $3$-Sasakian structure is a Riemannian structure whose metric cone is
hyperk\"ahler.  The sphere $\Ss_r^{4n-1}\subset\Hh^n$ carries the standard
example \cite{BGP}.  We use the convention adapted to the
diagonal left $\Sp(1)$ action.  Let
\[
I_1=i,\qquad I_2=j,\qquad I_3=k
\]
act by left quaternionic multiplication and put
\[
\xi_\alpha(q)=I_\alpha q,
\qquad
\eta_\alpha(V)=\frac{1}{r^2}
\langle V,I_\alpha q\rangle,
\qquad
\alpha=1,2,3.
\]
The three vector fields $\xi_\alpha$ generate the diagonal left
$\Sp(1)$ action.  The quaternionic contact distribution is
\[
\mathcal H_q
=
\bigcap_{\alpha=1}^3\ker\eta_\alpha
\subset T_q\Ss_r^{4n-1}.
\]
On $\mathcal H$ the three forms $d\eta_\alpha$ are the fundamental
$2$-forms of the standard quaternionic contact structure, up to the
normalization of the metric; compare \cite{IMV}.

\begin{theorem}[Horizontal and isotropic trace leaves]
\label{thm:q-qc-isotropic}
Every trace leaf of the Poincar\'e foliation is tangent to
$\mathcal H$.  Moreover, for every $\alpha=1,2,3$,
\[
d\eta_\alpha|_{T\widehat{\mathcal L}}=0.
\]
Equivalently, the tangent space of a trace leaf is simultaneously
isotropic for the three fundamental $2$-forms of the standard
$3$-Sasakian structure.
\end{theorem}

\begin{proof}
A tangent vector to a trace leaf has the form
\[
Y=(c_1q_1,\ldots,c_nq_n),
\qquad c_i\in\Rr,
\]
with
\[
\sum_i c_i|q_i|^2=0
\]
because $Y$ is tangent to the sphere.  For every imaginary unit
$I_\alpha$,
\[
\eta_\alpha(Y)
=
\frac1{r^2}
\sum_i c_i\langle q_i,I_\alpha q_i\rangle
=0,
\]
since $\langle q,I_\alpha q\rangle=0$.  Thus
$Y\in\mathcal H$.

Differentiating the definition of $\eta_\alpha$ gives, for horizontal
vectors,
\[
d\eta_\alpha(Y,Z)
=
\frac{2}{r^2}\langle I_\alpha Y,Z\rangle
=
-\frac{2}{r^2}\langle Y,I_\alpha Z\rangle.
\]
If
\[
Y=(c_iq_i),\qquad Z=(d_iq_i)
\]
are tangent to the same trace leaf, then
\[
\langle Y,I_\alpha Z\rangle
=
\sum_i c_id_i\langle q_i,I_\alpha q_i\rangle
=0.
\]
Hence $d\eta_\alpha(Y,Z)=0$ for all $\alpha$.
\end{proof}

Let
\[
\pi:\Ss_r^{4n-1}\longrightarrow\Hh P^{n-1}
\]
be the Hopf projection associated with the diagonal left $\Sp(1)$ action.
Since the tangent spaces of the trace leaves are horizontal,
$d\pi$ is injective on them.

\begin{corollary}[Projection to quaternionic projective space]
\label{cor:q-totally-real}
The restriction of $\pi$ to every trace leaf is an isometric immersion
for the submersion metric.  Its image is totally real with respect to the quaternionic structure
of $\Hh P^{n-1}$.  If $I_1,I_2,I_3$ is a local quaternionic frame and
$\omega_\alpha(V,W)=g(I_\alpha V,W)$ are the corresponding fundamental
$2$-forms, then
\[
\omega_\alpha|_{d\pi(T\widehat{\mathcal L})}=0,
\qquad \alpha=1,2,3.
\]
In particular the Kraines $4$-form vanishes on four tangent vectors of
a projected trace leaf.
\end{corollary}

\begin{proof}
The Hopf map is a Riemannian submersion on the horizontal distribution.
The horizontal pullbacks of the local fundamental $2$-forms are
multiples of $d\eta_\alpha$.  The assertion follows from
Theorem~\ref{thm:q-qc-isotropic}.  The Kraines form is, locally and up
to normalization,
\[
\Omega_{\mathrm K}
=
\omega_1\wedge\omega_1+
\omega_2\wedge\omega_2+
\omega_3\wedge\omega_3,
\]
so its restriction to the projected leaf is zero.
\end{proof}

\begin{remark}
For the LVMQ range $n>4m$, a principal trace leaf has dimension
$4m-1<n-1$, which is compatible with the usual dimension bound for a
totally real subspace of a quaternionic Hermitian vector space of
quaternionic dimension $n-1$.
\end{remark}

\subsection{Metric geometry of the real trace leaves}

Fix a support $I$ and phases $u_i\in \Ss^{3}$.  On the sphere write
\[
q_i=r\sqrt{x_i}\,u_i,
\qquad x_i>0,
\qquad \sum_{i\in I}x_i=1.
\]
Let $\xi\in C_\Lambda$ be the separating vector used to define the trace
action, and put
\[
b_i=\langle\lambda_i,\xi\rangle>0.
\]
For $U\in H_\xi$, the infinitesimal radial correction is determined by
keeping the norm constant.  If
\[
\bar\lambda_U=\sum_{i\in I}x_i\langle\lambda_i,U\rangle,
\qquad
\bar b=\sum_{i\in I}x_i b_i,
\]
then
\[
\dot\tau_U=-\frac{\bar\lambda_U}{\bar b}.
\]

\begin{theorem}[Metric on a real trace leaf]\label{thm:q-trace-metric}
For $U,V\in H_\xi$, the round metric of $\Ss_r^{4n-1}$ induces on the
trace leaf the bilinear form
\[
g^{\rm tr}_q(U,V)
=
r^2\sum_{i\in I}x_i
\left(\lambda_U(i)-b_i\frac{\bar\lambda_U}{\bar b}\right)
\left(\lambda_V(i)-b_i\frac{\bar\lambda_V}{\bar b}\right),
\]
where $\lambda_U(i)=\langle\lambda_i,U\rangle$.  Its kernel is exactly
the infinitesimal isotropy of the trace action,
\[
\{U\in H_\xi:\ U+c\xi\in K_I\text{ for some }c\in\Rr\}.
\]
Consequently it is positive definite on the tangent space of the trace
leaf, whose dimension is $\varrho(I)-1$.
\end{theorem}

\begin{proof}
Differentiate the norm constraint for
$\Phi_{U+\tau(U,q)\xi}(q)$.  At the given point this gives
\[
0=2\sum_i x_i\bigl(\lambda_U(i)+\dot\tau_U b_i\bigr),
\]
which is the displayed formula for $\dot\tau_U$.  Since
$q_i=r\sqrt{x_i}u_i$ with fixed phases, the radial part of the round
metric is
\[
ds^2=\frac{r^2}{4}\sum_i\frac{dx_i^2}{x_i}.
\]
Moreover
\[
d\log x_i[U]
=2\left(\lambda_U(i)-b_i\frac{\bar\lambda_U}{\bar b}\right).
\]
Substitution gives the metric formula.

The quadratic form vanishes exactly when
$\lambda_U(i)=c b_i$ for all $i\in I$, where
$c=\bar\lambda_U/\bar b$.  Equivalently
$U-c\xi\in K_I$, which is precisely the infinitesimal isotropy of the
trace action.  The dimension statement agrees with
Theorem~\ref{thm:q-trace-action}.
\end{proof}

The radially normalized ambient orbit has a related Hessian metric. Define
\[
\widehat\Phi_T(q)=r\,\frac{\Phi_T(q)}{\|\Phi_T(q)\|}.
\]
Then the squared norms are
\begin{equation}\label{eq:exp-family}
x_i(T)
=
\frac{x_i(0)e^{2\langle\lambda_i,T\rangle}}
{\displaystyle\sum_{j\in I}x_j(0)e^{2\langle\lambda_j,T\rangle}}.
\end{equation}
Set
\[
\varphi(T)=\log\left(\sum_{i\in I}
x_i(0)e^{2\langle\lambda_i,T\rangle}\right)
\]
and
\[
K_I^{\rm norm}
=
\{U\in\Rr^p:\langle\lambda_i,U\rangle
\text{ is independent of }i\in I\}.
\]

\begin{proposition}[Hessian metric after radial normalization]
\label{prop:q-hessian-normalized}
The metric induced on the radially normalized orbit is
\[
g_T^{\rm norm}(U,V)
=
r^2\operatorname{Cov}_{x(T)}(\lambda_U,\lambda_V)
=
\frac{r^2}{4}D^2\varphi_T(U,V).
\]
Its kernel is exactly $K_I^{\rm norm}$, so it is positive definite on
$\Rr^p/K_I^{\rm norm}$.
\end{proposition}

\begin{proof}
Along~\eqref{eq:exp-family},
\[
d\log x_i(T)[U]
=2\left(\lambda_U(i)-\sum_jx_j(T)\lambda_U(j)\right).
\]
Substitution into
$ds^2=(r^2/4)\sum_i dx_i^2/x_i$ gives the covariance formula, and direct
differentiation of $\varphi$ gives the Hessian formula.  A weighted
variance vanishes exactly when the values $\lambda_U(i)$ are constant on
$I$, which is the stated kernel.
\end{proof}

\begin{remark}
The normalized orbit in Proposition~\ref{prop:q-hessian-normalized} need
not coincide with the trace leaf, because radial normalization by a common
scalar need not be generated by the real weight action.  They coincide on
the support $I$ if there exists $R\in\Rr^p$ such that
$\langle\lambda_i,R\rangle=1$ for every $i\in I$; in that special case a
common radial rescaling is an action direction and the Hessian metric
restricts to the trace leaf after quotienting that direction.  No such
assumption is made in general.
\end{remark}

\subsection{Transverse differential forms}

A differential form descends to the local leaf space only if it is
basic: it must vanish under contraction with every leafwise vector
field and be invariant along the leaves.  The contact forms satisfy
\[
\iota_Y\eta_\alpha=0
\]
for every leafwise $Y$. On a regular region with positive-dimensional
trace leaves, they are not invariant along the leaves.

\begin{proposition}\label{prop:q-not-basic}
Let $Y$ be a nonvanishing local leafwise vector field tangent to a trace
leaf.  Then, for each $\alpha$, there is a horizontal vector $Z$ with
$d\eta_\alpha(Y,Z)\neq0$. Consequently $\mathcal L_Y\eta_\alpha$
restricted to $\mathcal H$ is not zero, so $\mathcal L_Y\eta_\alpha$
does not belong to the span of $\eta_1,\eta_2,\eta_3$ over the smooth
functions, and the flow of $Y$ does not preserve $\mathcal H$.
Therefore the standard quaternionic contact structure does not descend
to the local trace leaf space on any regular region with
positive-dimensional leaves.
\end{proposition}

\begin{proof}
Since $Y$ is horizontal, so is $I_\alpha Y$.  Taking
$Z=I_\alpha Y$ and differentiating $\eta_\alpha$,
\[
d\eta_\alpha(Y,Z)
=
\frac{2}{r^2}
\langle I_\alpha Y,I_\alpha Y\rangle
=
\frac{2}{r^2}|Y|^2\neq0.
\]
The identity $\iota_Y\eta_\alpha=0$ holds identically for a local
leafwise vector field $Y$, so Cartan's formula gives
$\mathcal L_Y\eta_\alpha=\iota_Yd\eta_\alpha$, whose value on $Z$ is
nonzero. Every function combination of $\eta_1,\eta_2,\eta_3$ vanishes
on $\mathcal H$.
\end{proof}

The same observation applies to the quaternionic K\"ahler $4$-form on
$\Hh P^{n-1}$.  Although its restriction to four tangent vectors of a
projected trace leaf vanishes, its contraction by one nonzero leafwise
vector does not vanish identically.  Hence the Kraines form is not
basic on any regular region with positive-dimensional projected trace
leaves and does not furnish a transverse $4$-plectic form there.

\begin{corollary}\label{cor:q-noautomatic}
On any open region where the Poincar\'e trace foliation is regular
and has positive-dimensional leaves, the standard quaternionic contact
structure of $\Ss^{4n-1}$ does not descend to the local leaf space.
\end{corollary}

If all weights lie on one positive ray, the ambient orbits have dimension
one and the trace leaves are points. Their leaf space is the sphere
itself, with its quaternionic contact structure.

\begin{problem}
Determine whether the regular principal leaf space
\[
\Rr^{n-4m}\times(\Ss^{3})^n
\]
admits a natural $4$-plectic form constructed from the weight
configuration and the quaternionic phase factors, and determine
whether such a form can be extended compatibly across the lower
support strata.
\end{problem}

\begin{problem}
Study the projected trace leaves in $\Hh P^{n-1}$ as totally real
submanifolds.  In particular, determine their second fundamental form,
minimality properties, and the configurations for which some of these
leaves are totally geodesic or minimal.
\end{problem}

\subsection{A quaternionic rank-one distribution}
\label{subsec:q-rank-one}

The preceding trace foliation has Euclidean leaves because it comes from
an abelian real vector group.  A rank-one quaternionic distribution is
defined differently.

Let $U\subset\Hh^n$ be open, $n\ge2$, and let $X$ be a smooth vector
field on $U$.  Using the standard trivialization
$T\Hh^n\simeq\Hh^n\times\Hh^n$, regard $X$ as a smooth map
$U\to\Hh^n$.  At a point where $X(q)\neq0$, set
\[
\mathcal D_X(q)=X(q)\Hh
=\{X(q)a:a\in\Hh\}
=\operatorname{span}_{\Rr}\{X(q),X(q)i,X(q)j,X(q)k\}.
\]
This is a real four-plane distribution.  Only when $\mathcal D_X$ is
involutive does it define a four-dimensional foliation.  If, in addition,
its leaves are transverse to a sphere, the induced trace leaves have real
dimension three.

The basic example is the radial field
\[
R(q_1,\ldots,q_n)=(q_1,\ldots,q_n).
\]

In this rank-one discussion the Hopf quotient is the usual
\emph{right}-projective quaternionic space
\[
(\Hh^n\setminus\{0\})/\Hh^\ast ,
\qquad q\sim qa.
\]
It is diffeomorphic to the left-projective convention used earlier for
the LVMQ quotient.  We denote both models by $\Hh P^{n-1}$, but the
side of the scalar action will be stated when it matters.

\begin{theorem}[Quaternionic Hopf trace foliation]
\label{thm:hopf-trace}
On $\Hh^n\setminus\{0\}$ the distribution
\[
\mathcal D_R(q)=q\Hh
\]
is integrable.  Its leaves are the punctured right quaternionic lines
$q\Hh^\ast$.  Their intersections with $\Ss^{4n-1}$ are
\[
q\Sp(1)\cong \Ss^{3}.
\]
Consequently
\[
\Ss^{3}\longrightarrow \Ss^{4n-1}\longrightarrow\Hh P^{n-1}
\]
is exactly the trace foliation.  It is regular on the sphere; there are
no singular spherical leaves.  In $\Hh^n$ the origin is the unique
singular leaf.
\end{theorem}

\begin{proof}
Right multiplication by $\Hh^\ast$ is a free action on
$\Hh^n\setminus\{0\}$, and its orbits are the punctured right
quaternionic lines.  Restriction to the unit sphere replaces
$\Hh^\ast\simeq\Rr_{>0}\times\Sp(1)$ by $\Sp(1)$, giving the standard
quaternionic Hopf fibration.
\end{proof}

Every punctured Hopf leaf is $\Hh^*\cong\Rr_{>0}\times\Ss^3$.
Its spherical trace is a compact embedded $\Ss^3$ with trivial
holonomy. The quotient is Hausdorff because the $\Sp(1)$ action is
free and proper.

\subsubsection*{Hopf quotients and the Calabi--Eckmann analogy}

The scalar rank-one model also contains the usual compact Hopf quotient.
Fix $0<\rho<1$ and let
\[
\gamma_\rho(q)=\rho q
\]
on $\Hh^n\setminus\{0\}$.  The cyclic group generated by $\gamma_\rho$
acts freely and properly discontinuously.

\begin{proposition}\label{prop:cyclic-hopf-quotient}
For every $n\ge1$ and every $0<\rho<1$,
\[
(\Hh^n\setminus\{0\})/\langle\gamma_\rho\rangle
\cong
\Ss^{4n-1}\times \Ss^1 .
\]
\end{proposition}

\begin{proof}
Write $q=r u$ with $r>0$ and $u\in\Ss^{4n-1}$.  Under the logarithmic
coordinate $t=\log r$, the generator acts by
\[
(u,t)\longmapsto (u,t+\log\rho).
\]
Thus the quotient is
\[
\Ss^{4n-1}\times
\bigl(\Rr/(\log\rho)\Zz\bigr)
\cong
\Ss^{4n-1}\times\Ss^1 .
\]
\end{proof}

For $n=2$ this gives $\Ss^7\times\Ss^1$, the underlying smooth manifold
on which Angella and Bisi construct slice-quaternionic Hopf surfaces
\cite{AngellaBisi}.  This quotient is not an LVMQ manifold: every LVMQ
manifold is $2$-connected, whereas the cyclic Hopf quotient has fundamental
group $\Zz$.

The sphere-product family of Theorem~\ref{thm:einstein-family} gives
\[
N_{\Hh}(\Lambda)
\cong
\Ss^{4d+3}\times(\Ss^3)^{4m-1}.
\]
In complex LVM theory the corresponding sphere products are the
Calabi--Eckmann manifolds. We do not know whether these LVMQ sphere
products carry slice-quaternionic or hypercomplex structures
compatible with the construction.

\subsubsection*{Linear Frobenius rigidity in every quaternionic dimension}

We use the convention that
$A\in M_n(\Hh)$ acts on column vectors from the left; then
$A(qa)=(Aq)a$, so $A$ is right-$\Hh$-linear.

\begin{theorem}[Linear Frobenius rigidity]
\label{thm:linear-frobenius-rigidity}
Let
\[
X(q)=Aq,\qquad A\in GL(n,\Hh),\qquad n\ge2,
\]
and let $\mathcal D_X(q)=Aq\,\Hh$.  Then $\mathcal D_X$ is involutive on
$\Hh^n\setminus\{0\}$ if and only if
\[
A=\mu I,
\qquad \mu\in\Rr^\ast.
\]
Hence every right-$\Hh$-linear isolated-singularity model with an
involutive quaternionic rank-one distribution has on $\Ss^{4n-1}$
exactly the Hopf foliation by $\Ss^{3}$'s.
\end{theorem}

\begin{proof}
For $a\in\Hh$ put $X_a(q)=Aq\,a$.  Since $A$ is right-$\Hh$-linear,
\[
[X_a,X_b](q)=A^2q\,(ab-ba)
\]
for the convention $[V,W]=DW(V)-DV(W)$.  Frobenius integrability
therefore implies
\[
A^2q\,c\in Aq\,\Hh
\]
for every $q$ and every purely imaginary quaternion $c$ which occurs as
a commutator.  Choose one nonzero such $c$.  Since both $A$ and $c$ are
invertible, putting $p=Aq$ gives
\[
Ap\in p\Hh
\qquad\text{for every }p\neq0.
\]
Thus $A$ preserves every right quaternionic line.

Write
\[
Ae_r=e_r\alpha_r
\]
for the standard basis.  Applying the preceding property to $e_r+e_s$
shows that $\alpha_r=\alpha_s$ for all $r,s$; let their common value be
$\alpha$.  Applying it to $e_1+e_2h$, with arbitrary $h\in\Hh$, gives
\[
\alpha h=h\alpha
\qquad\text{for every }h\in\Hh.
\]
Therefore $\alpha$ belongs to the center of $\Hh$, namely $\Rr$, and
$A=\mu I$.  Conversely, for $A=\mu I$ the distribution is $q\Hh$ and
is the orbit distribution of right multiplication by $\Hh^\ast$.
\end{proof}

In particular, for a real diagonal matrix with two distinct nonzero
weights, the four-plane field fails the Frobenius condition at every
point having nonzero coordinates with distinct weights. The
coordinate-axis spheres $\Ss^{3}$ are integral submanifolds of this
field, but the field defines no foliation of the sphere.

\subsubsection*{Several quaternionic vector fields}

Let $k\ge2$ and let
\[
X_1,\ldots,X_k
\]
be smooth vector fields on an open subset of $\Hh^n$, and define
\[
\mathcal D(q)
=
X_1(q)\Hh+\cdots+X_k(q)\Hh.
\]
On the locus where the vectors $X_1(q),\ldots,X_k(q)$ are
right-$\Hh$-linearly independent, $\mathcal D$ has real rank $4k$.
For $k>1$ the involutivity of $\mathcal D$ is weaker than the
involutivity of each individual rank-four distribution
$X_\alpha\Hh$.  The one-field rigidity theorem therefore does not apply component by
component.

For the linear family
\[
X_\alpha(q)=A_\alpha q,
\qquad
A_\alpha\in M_n(\Hh),
\qquad
1\le\alpha\le k,
\]
put
\[
X_{\alpha,a}(q)=A_\alpha q\,a,
\qquad a\in\Hh.
\]
A direct computation gives
\[
[X_{\alpha,a},X_{\beta,b}](q)
=
A_\beta A_\alpha q\,(ab)
-
A_\alpha A_\beta q\,(ba),
\]
for the same bracket convention.

\begin{proposition}[Frobenius criterion for a linear quaternionic family]
\label{prop:multi-frobenius}
On a constant-rank stratum, the distribution
\[
\mathcal D(q)
=
\sum_{\alpha=1}^k A_\alpha q\,\Hh
\]
is involutive if and only if, for every $\alpha,\beta$, every
$a,b\in\Hh$, and every $q$ in that stratum,
\[
A_\beta A_\alpha q\,(ab)
-
A_\alpha A_\beta q\,(ba)
\in
\sum_{\gamma=1}^k A_\gamma q\,\Hh.
\]
\end{proposition}

\begin{proof}
The fields $X_{\alpha,a}$ span $\mathcal D$ over $\Rr$.  Frobenius'
criterion is therefore exactly the requirement that all their pairwise
brackets lie in the same distribution.  The displayed bracket formula
gives the stated condition.
\end{proof}

Even when $\mathcal D$ is involutive, one need not have
\[
[A_\alpha q\,\Hh,A_\alpha q\,\Hh]
\subset A_\alpha q\,\Hh
\]
for each $\alpha$ separately: the bracket may close only after the
other quaternionic directions are included.  Thus the conclusion
$A_\alpha=\mu_\alpha I$ is not valid for a general multi-field system.

\begin{proposition}\label{prop:multi-invertible-obstruction}
Assume $k\ge2$. Then there is $q\neq0$ for which
\[
\{A_1q,\ldots,A_kq\}
\]
has right quaternionic rank at most $k-1$.
\end{proposition}

\begin{proof}
If some $A_\alpha$ is singular, take $q\ne0$ with $A_\alpha q=0$.
Otherwise $A_1$ is invertible; set
\[
B_\alpha=A_\alpha A_1^{-1}.
\]
Then $B_1=I$.  Every quaternionic matrix has a right eigenvalue; see
\cite{ZhangQuaternionMatrices}.  Hence there exist $p\neq0$ and
$\lambda\in\Hh$ such that
\[
B_2p=p\lambda.
\]
With $q=A_1^{-1}p$ one obtains
\[
A_2q=A_1q\,\lambda.
\]
Thus the first two vectors in
$\{A_1q,\ldots,A_kq\}$ are right-$\Hh$-linearly dependent, and the
quaternionic rank is at most $k-1$ at that point.
\end{proof}

Since $A_\alpha(qt)=(A_\alpha q)t$ for real $t>0$, a linear family
with $k\ge2$ never has constant rank $4k$ on a punctured neighborhood
of the origin. If a quaternionic rank-$k$ distribution is involutive,
Lemma~\ref{lem:flat-quaternionic} below places each integral leaf in an
affine quaternionic $k$-plane. The classification of the resulting
families of affine planes remains to be studied.

\begin{problem}[Higher-rank quaternionic Poincar\'e problem]
For $k\ge2$, classify nonlinear families $X_1,\ldots,X_k$ for which
\[
\mathcal D=X_1\Hh+\cdots+X_k\Hh
\]
has constant rank $4k$ on a punctured neighborhood, is involutive, and
contains a direction transverse to sufficiently small centered spheres.
Describe the induced $(4k-1)$-dimensional trace foliation and its leaf
space.  Determine which additional hypotheses, if any, force compact
leaves or a fibration.
\end{problem}

\subsubsection*{Nonlinear rigidity near the origin}

We first recall the flat quaternionic submanifold argument. An immersed
submanifold of $\Hh^n$ is called quaternionic here if every tangent
space is invariant under right multiplication by $i,j,k$.

\begin{lemma}\label{lem:flat-quaternionic}
Every connected $C^2$ quaternionic submanifold of $\Hh^n$ has zero
second fundamental form. Its image is contained in an affine
quaternionic subspace, and locally it is an open subset of that subspace.
\end{lemma}

\begin{proof}
Let $B(u,v)=(\nabla_u v)^\perp$ be the second fundamental form for
the Euclidean connection. The tangent and normal spaces are invariant
under right multiplication by $i,j,k$. Since these multiplications
are parallel, for $a\in\{i,j,k\}$ we have
\[
B(u,va)=B(u,v)a.
\]
Symmetry of $B$ also gives $B(ua,v)=B(u,v)a$. Applying these identities
in the two orders yields
\[
B(ui,vj)=(B(u,v)i)j=B(u,v)k,
\]
whereas
\[
B(ui,vj)=(B(u,v)j)i=-B(u,v)k.
\]
Thus $B=0$. Parallel translation along a path in the submanifold
preserves its tangent space. Euclidean parallel translation is
constant, so connectedness makes this tangent space a fixed
quaternionic subspace $W$. Integration along paths places the image
in one translate of $W$, and the inverse function theorem gives the
local assertion.
\end{proof}

\begin{theorem}[Nonlinear Hopf rigidity]\label{thm:hopf-like-poincare}
Let $X$ be a $C^2$ real vector field near $0\in\Hh^n$, $n\ge2$, with
\[
X(0)=0,\qquad DX(0)=\mu I,\qquad \mu>0.
\]
Suppose $\mathcal D_X(q)=X(q)\Hh$ is involutive on a punctured
neighborhood of the origin. Then there is $\varepsilon>0$ such that
\[
\mathcal D_X(q)=q\Hh\qquad(0<\|q\|<\varepsilon).
\]
For every $0<r<\varepsilon$, its trace on $\Ss_r^{4n-1}$ is exactly
the standard right Hopf foliation. In particular the quotient is the
standard smooth manifold $\Hh P^{n-1}$.
\end{theorem}

\begin{proof}
Choose $\varepsilon>0$ so small that $X$ is defined on a neighborhood
of the closed ball and
\[
\langle X(q),q\rangle\ge \frac\mu2\|q\|^2,
\qquad 0<\|q\|\le\varepsilon.
\]
This follows from $X(q)=\mu q+O(\|q\|^2)$, and it implies $X(q)\ne0$.
Hence $\mathcal D_X$ has constant real rank four on the punctured ball.
Frobenius' theorem gives integral leaves; each is a quaternionic
submanifold and is contained in an affine quaternionic line by
Lemma~\ref{lem:flat-quaternionic}.

Fix $q$ in the punctured ball and let $\phi_t(q)$ be its trajectory
under $X$. It exists for all $t\le0$ and stays in the ball: the norm
strictly decreases in backward time, and the vector field is defined
on a neighborhood of the closed ball. Uniqueness prevents a nonzero
trajectory from reaching the equilibrium $0$ in finite time. Moreover,
\[
\frac{d}{dt}\|\phi_t(q)\|^2\ge\mu\|\phi_t(q)\|^2
\]
gives
\[
\|\phi_t(q)\|^2\le e^{\mu t}\|q\|^2\qquad(t\le0).
\]
Thus $\phi_t(q)\to0$ as $t\to-\infty$.

The trajectory stays in the integral leaf through $q$, since $X$ is
tangent to that leaf. Its containing affine quaternionic line is
closed in $\Hh^n$ and therefore contains $0$. It is consequently the
linear quaternionic line $q\Hh$. The leaf has real dimension four,
so its tangent space at $q$ is $q\Hh$. This holds at every point of
the punctured ball and proves the stated equality of distributions.

The maximal leaves in the ball are the connected punctured disks
$(q\Hh)\cap(B_\varepsilon\setminus\{0\})$. Their intersections with
a centered sphere are the right $\Sp(1)$ orbits. The quotient map is
the usual smooth Hopf projection.
\end{proof}

\begin{remark}
The vector field itself can be nonlinear. For example,
$X(q)=q\,a(q)$, where $a$ is $C^2$ and $a(0)=\mu>0$, has
$X(q)\Hh=q\Hh$ wherever $a(q)\ne0$. The theorem determines the
quaternionic distribution and its leaves, without asserting linear
conjugacy of the one-dimensional flow of $X$.
\end{remark}

\subsection{The Poincar\'e--Siegel dichotomy for the real action}

\begin{theorem}[Poincar\'e--Siegel dichotomy for the canonical real action]\label{thm:q-dichotomy}
Let
\[
\Lambda=(\lambda_1,\ldots,\lambda_n)
\subset\Hh^m\simeq\Rr^{4m}
\]
be of full affine rank, and let $\Phi$ be the real action
\eqref{eq:poincare-action}.

\begin{enumerate}[label=\textnormal{(\roman*)}]
\item If
\[
0\in\conv\Lambda
\]
and the weak hyperbolicity condition holds, then on the Siegel open
set every orbit has a unique point of minimum norm.  The normalized
minimum set is $Z_{\Hh}(\Lambda)$, and
\[
\cS_\Lambda
\cong
Z_{\Hh}(\Lambda)\times\Rr^{4m}\times\Rr_{>0}.
\]
After quaternionic projectivization the leaf space is the compact
LVMQ manifold $N_{\Hh}(\Lambda)$.

\item If
\[
0\notin\conv\Lambda,
\]
then there exists $\xi\in\Rr^{4m}$ for which the norm is strictly
increasing along every $\Rr\xi$ orbit.  Every sphere is a global
section for this one-dimensional subaction and
\[
\Hh^n\setminus\{0\}
\cong
\Ss_r^{4n-1}\times\Rr.
\]
The residual action on the sphere defines a singular trace foliation
whose leaf in the support stratum $I$ is $\Rr^{\varrho(I)-1}$ and whose
stratum leaf space is
\[
\Rr^{|I|-\varrho(I)}\times(\Ss^{3})^{|I|}.
\]
\end{enumerate}
\end{theorem}

The case in which the origin lies on the boundary of
$\conv\Lambda$ is not treated here.

\section{Cohomology and further questions}
\label{sec:further-results}

\subsection{The integral cohomology ring}

The additive (co)homology of $Z_{\Hh}(\Lambda)$ is given by
Theorem~\ref{thm:hochster}. There is also a diagonal-subgroup description
of the quotient.

\begin{proposition}\label{prop:conjugation}
Let $G=\Sp(1)^n$ act on $Z_{\Hh}(\Lambda)$ on the right, and let
$\Delta\cong\Sp(1)$ be its diagonal subgroup. Coordinatewise
quaternionic conjugation induces a diffeomorphism
\[
N_{\Hh}(\Lambda)\cong Z_{\Hh}(\Lambda)/\Delta.
\]
\end{proposition}

\begin{proof}
The map $C(q)=(\overline q_1,\ldots,\overline q_n)$ preserves all
coordinate norms, hence preserves $Z_{\Hh}(\Lambda)$. It is a smooth
involution and satisfies
\[
C(aq)=C(q)\,\overline a.
\]
It therefore identifies the left diagonal orbits with the right
diagonal orbits and induces the stated diffeomorphism.
\end{proof}

In particular, $N_{\Hh}$ has the homotopy type of the Borel
construction $E\Delta\times_\Delta Z_{\Hh}$, because the diagonal
action is free. Here $E\Delta$ is a contractible space with a free
$\Delta$ action and $B\Delta=E\Delta/\Delta$ is its classifying space.
The Borel construction consists of the orbits in
$E\Delta\times Z_{\Hh}$ for the diagonal action. It is obtained,
up to homotopy, by pulling back
\[
EG\times_G Z_{\Hh}\longrightarrow BG
\]
along $B\Delta\to BG$. Thus restriction to the diagonal subgroup
provides an equivariant approach to the quotient cohomology.
For comparison, quotient cohomology for complex moment-angle spaces
and subtori is treated in \cite{LMM}.

\begin{problem}
Determine the integral cohomology ring of $N_{\Hh}(\Lambda)$ from
$K_{P_\Lambda}$ using this diagonal pullback, including the products
and torsion not determined by the Gysin sequence.
\end{problem}

\subsection{Dependence on the labelled combinatorial type}

Theorem~\ref{thm:combinatorial} gives topological invariance under
labelled combinatorial equivalence. A smooth version requires
compatible smooth corner structures and equivariant smoothings of the
polyhedral-product models.

\begin{problem}
Determine whether labelled combinatorial equivalence of the associated
simple polytopes, with indispensable coordinates matched, always
determines the diffeomorphism type of the corresponding LVMQ manifolds.
\end{problem}

\subsection{The smooth wall-crossing problem}

Theorem~\ref{thm:wall} gives the combinatorial flip of the associated
simple polytope.  For the quaternionic moment-angle total space the
corresponding local replacement is expected to have the form
\[
\Ss^{4a-1}\times(\Ss^{3})^{4m}\times D^{4b}
\quad\longleftrightarrow\quad
D^{4a}\times(\Ss^{3})^{4m}\times \Ss^{4b-1}.
\]
To obtain a smooth surgery theorem for the LVMQ quotient one must
construct the equivariant tubular neighborhoods and verify the gluing
map with the diagonal left $\Sp(1)$ action.  These data are not
determined by the combinatorial flip alone.

\begin{problem}\label{prob:wall}
Give an explicit diagonal-$\Sp(1)$-equivariant local model for a generic
wall crossing and determine the induced smooth surgery on
$N_{\Hh}(\Lambda)$.
\end{problem}

\section*{Acknowledgments}

I thank Aubin Arroyo, Carlos Cabrera, and Jos\'e Seade for discussions
and for our collaboration on holomorphic linear actions and
Poincar\'e dynamics, which suggested questions considered here.
I also thank Laurent Meersseman and Daniele Zuddas for discussions
on the subjects of this paper.

\section*{Funding}

The author acknowledges {\bf Proyecto PAPIIT IN103324
(DGAPA, UNAM, M\'exico)} for its financial support.

\medskip
\noindent\textbf{Use of generative AI.}
The author used ChatGPT (OpenAI) and Claude (Anthropic) to assist with
editing, calculations, and the examination of mathematical arguments.
The author is responsible for the final manuscript.

\end{document}